\documentclass[11pt]{amsart}
\usepackage{amssymb}
\usepackage{amscd}
\usepackage{amsmath}
\usepackage[all]{xy}
\usepackage{mathrsfs}
 \usepackage{color}

\numberwithin{equation}{subsection}

\calclayout

\theoremstyle{plain}
\newtheorem{theorem}{Theorem}[section]
\newtheorem{theorem*}{Theorem}
\newtheorem{proposition}[theorem]{Proposition}
\newtheorem{lemma}[theorem]{Lemma}
\newtheorem{corollary}[theorem]{Corollary}
\newtheorem{conjecture}[theorem]{Conjecture}

\theoremstyle{remark}
\newtheorem{remark}[theorem]{Remark}

\theoremstyle{definition}
\newtheorem{definition}[theorem]{Definition}
\newtheorem{example}[theorem]{Example}
\newtheorem{hypothesis}[theorem]{Hypothesis}

\newtheorem*{acknowledgments}{Acknowledgements}

\font\russ=wncyr10  1
\def\sha{\hbox{\russ\char88}}

\DeclareMathOperator{\Gal}{Gal}

\DeclareMathOperator{\Hom}{Hom}

\DeclareMathOperator{\im}{im}
\DeclareMathOperator{\coker}{coker}

\newcommand{\bA}{\mathbb{A}}

\newcommand{\TT}{\mathbb{T}}

\newcommand{\QQ}{\mathbb{Q}}

\newcommand{\cD}{\mathcal{D}}

\newcommand{\cL}{\mathcal{L}}

\newcommand{\cO}{\mathcal{O}}

\newcommand{\cQ}{\mathcal{Q}}
\newcommand{\cR}{\mathcal{R}}

\newcommand{\fq}{\mathfrak{q}}
\newcommand{\fp}{\mathfrak{p}}

\newcommand{\fz}{\mathfrak{z}}
\newcommand{\fQ}{\mathfrak{Q}}

\newcommand{\CC}{\mathbb{C}}

\newcommand{\FF}{\mathbb{F}}
\newcommand{\GG}{\mathbb{G}}

\newcommand{\RR}{\mathbb{R}}

\newcommand{\ZZ}{\mathbb{Z}}

\newcommand{\sL}{\mathscr{L}}

\newcommand{\rgamma}{\mathbf{R}\Gamma}
\newcommand{\rhom}{\mathbf{R}\Hom}
\newcommand{\lotimes}{\otimes^{\mathbf{L}}}

\begin{document}

\title[]{Derived Bockstein regulators and anticyclotomic $p$-adic Birch and Swinnerton-Dyer conjectures}

\author{Takamichi Sano}

\begin{abstract}
We introduce ``derived Bockstein regulators" by using an idea of Nekov\'a\v r. We establish a general descent formalism involving derived Bockstein regulators. We give three applications of this formalism. Firstly, we show that a conjecture of Birch and Swinnerton-Dyer type for Heegner points formulated by Bertolini and Darmon in 1996 follows from Perrin-Riou's Heegner point main conjecture up to a $p$-adic unit. Secondly, we show that a $p$-adic Birch and Swinnerton-Dyer conjecture for the Bertolini-Darmon-Prasanna $p$-adic $L$-function recently formulated by Agboola and Castella follows from the Iwasawa-Greenberg main conjecture up to a $p$-adic unit. Finally, we extend conjectures and results on derivatives of Euler systems for a general motive given by Kataoka and the present author into a natural derived setting. 
\end{abstract}

\address{Osaka Metropolitan University,
Department of Mathematics,
3-3-138 Sugimoto\\Sumiyoshi-ku\\Osaka\\558-8585,
Japan}
\email{tsano@omu.ac.jp}

\maketitle

\tableofcontents

\section{Introduction}

\subsection{Background and main results}
In the pioneering work \cite{MT}, Mazur and Tate proposed a conjecture which refines the celebrated Birch and Swinnerton-Dyer conjecture. Their conjecture can be regarded as a ``tame analogue" of the $p$-adic Birch and Swinnerton-Dyer conjecture formulated earlier by Mazur, Tate and Teitelbaum in \cite{MTT}. Inspired by their work, Gross considered an analogue for $L$-functions over totally real number fields in \cite{Gp} and \cite{G}, which is often referred to as the ``Gross-Stark conjecture". 

In the series of works \cite{bks1}, \cite{bks2}, \cite{bks4} and \cite{bks5}, Burns, Kurihara and the present author intensively studied such refined conjectures and observed that the conjectures of Mazur-Tate, Mazur-Tate-Teitelbaum and Gross can be treated, in principle, by a single formalism. The key idea behind this formalism goes back to the systematic study of Bockstein maps and height pairings by Nekov\'a\v r in \cite[\S 11]{nekovar}. 

In the settings considered by Burns, Kurihara and the present author, natural $p$-adic regulators constructed by using Bockstein maps are expected to be non-zero. In other words, in these settings, $p$-adic height pairings are conjectured to be non-degenerate. However, as observed by Bertolini and Darmon in \cite{BDMT} and \cite{BD der}, {\it a natural $p$-adic height pairing is always degenerate in the anticyclotomic setting and in this case one needs to consider ``derived heights"}. 

The aim of the present work is to extend the Bockstein formalism of Burns, Kurihara and the present author into ``derived" settings. To achieve this, we develop the theory of the Bockstein spectral sequence by Nekov\'a\v r in \cite[\S 11.6]{nekovar}. Using ideas in loc. cit., we introduce ``derived Bockstein regulators" and study their algebraic properties. Our main algebraic result is to establish a descent formalism involving derived Bockstein regulators (see Theorem \ref{thm descent}). Our descent formalism is more general than that of Nekov\'a\v r in the sense that we do not assume that cohomology is $\Lambda$-torsion (see Remark \ref{more general}). 

In this paper, we give three applications of our formalism. Firstly, we show that {\it a conjecture of Birch and Swinnerton-Dyer type for Heegner points formulated by Bertolini and Darmon in \cite{BD} follows from Perrin-Riou's Heegner point main conjecture up to a $p$-adic unit} (see Theorem \ref{main}). Combining this result with the recent work of Burungale, Castella and Kim \cite{BCK} on the Heegner point main conjecture, we obtain an unconditional result on the conjecture of Bertolini and Darmon (see Corollary \ref{BCK}). As far as the author is aware, this is the first theoretical result on the conjecture in the higher rank case. 

Secondly, we show that {\it a $p$-adic Birch and Swinnerton-Dyer conjecture for the Bertolini-Darmon-Prasanna $p$-adic $L$-function recently formulated by Agboola and Castella in \cite{AC} follows from the Iwasawa-Greenberg main conjecture up to a $p$-adic unit} (see Theorem \ref{improved AC}). A similar result is obtained by Agboola and Castella in \cite[Theorem 6.2]{AC}, but we work under milder hypotheses: we do not assume that $p$ does not divide the Tamagawa factors or that $p$ is non-anomalous. 

Finally, {\it we give refinements of conjectures and results given by Kataoka and the present author in \cite{ks}, which generalize those of Burns, Kurihara and the present author in \cite{bks4} to a general setting of motives}. In particular, we formulate a conjecture on derivatives of Euler systems for a general motive by using derived Bockstein regulators (see Conjecture \ref{leading}). In a forthcoming work, we show that this conjecture essentially generalizes the conjecture of Bertolini and Darmon. As shown in \cite{bks2},  \cite{bks4} and \cite{bks5}, Conjecture \ref{leading} also generalizes the Mazur-Tate conjecture, the Mazur-Tate-Teitelbaum conjecture and the Gross-Stark conjecture (and, more generally, a conjecture of Mazur and Rubin \cite{MRGm} and the present author \cite{sano}). Thus, one may regard Conjecture \ref{leading} as a ``refined conjecture of Birch and Swinnerton-Dyer type for a general motive". Applying our descent formalism, we give natural descent results concerning Conjecture \ref{leading}: see Theorems \ref{strategy}, \ref{alg from IMC} and \ref{bsd strategy}.

It seems that our descent formalism has many other arithmetic applications. Although we focus on the anticyclotomic setting in this paper, our formalism can of course be used in the classical cyclotomic setting. For example, it is possible to show that the Mazur-Tate-Teitelbaum conjecture \cite{MTT} follows from the cyclotomic Iwasawa main conjecture for elliptic curves over $\QQ$ up to a $p$-adic unit by using our formalism (even in the split multiplicative case), and thus recover the classical results of Perrin-Riou \cite{PR} and Schneider \cite{schneider}, \cite{schneider2}. (Since cohomology is $\Lambda$-torsion in the cyclotomic case, this result is actually covered by Nekov\'a\v r's formalism given in \cite[\S 11.7]{nekovar}.) It would be interesting to apply our formalism also in Hida theoretic settings. 

Although we consider only Iwasawa theoretic settings in this paper, our formalism would be extended to a more general setting over finite group rings. In such a way, we expect that the conjecture of Darmon on Heegner points in \cite{DH}, which can be regarded as a tame analogue of the conjecture of Bertolini and Darmon, is also treated by a similar formalism. 

\begin{remark}
After the author wrote a first draft of this paper, Francesc Castella told him that Castella-Hsu-Kundu-Lee-Liu were working on the extension of the Agboola-Castella conjecture to supersingular primes. They obtain a similar result to our Theorem \ref{improved AC}, including the supersingular case (see \cite{CHKLL}). 
\end{remark}

\subsection{General idea}

For the reader's convenience, we explain the general idea of our formalism. 

Let $\Lambda$ be an Iwasawa algebra over $\ZZ_p$ (for simplicity)  and $Q(\Lambda)$ the total quotient ring of $\Lambda$. Suppose that a complex $C$ of finitely generated $\Lambda$-modules acyclic outside degrees one and two is given. (In practice, we take $C$ to be a Selmer complex.) Let $\sigma$ and $\tau$ be the $\Lambda$-ranks of $H^1(C)$ and $H^2(C)$ respectively. For a commutative ring $R$, let $\det_R$ denote the determinant functor. Then we have a canonical isomorphism
$$\pi: Q(\Lambda)\otimes_\Lambda {\det}^{-1}_\Lambda(C) \xrightarrow{\sim} Q(\Lambda) \otimes_\Lambda \left({\bigwedge}_\Lambda^\sigma H^1(C) \otimes_\Lambda {\bigwedge}_\Lambda^\tau H^2(C)^\ast\right),$$
where $(-)^\ast$ denotes the dual $\Hom_\Lambda(-,\Lambda)$. 
Suppose that a special element
$$z \in Q(\Lambda) \otimes_\Lambda \left({\bigwedge}_\Lambda^\sigma H^1(C) \otimes_\Lambda {\bigwedge}_\Lambda^\tau H^2(C)^\ast\right)$$
is given. The reader should think of this as a $p$-adic $L$-function (when $\sigma=\tau=0$) or an Euler system (in general). Consider the following condition:
$$\text{(IMC)} \quad \pi \left({\det}_\Lambda^{-1}(C)\right) = \Lambda\cdot z.$$
This condition corresponds to an Iwasawa main conjecture. For example, if $H^1(C)=0$ and $H^2(C)$ is $\Lambda$-torsion, then we have $\pi\left({\det}_\Lambda^{-1}(C)\right) = {\rm char}_\Lambda(H^2(C))$ and so this condition is the usual formulation of the Iwasawa main conjecture for a $p$-adic $L$-function. 

Let $I$ be the augmentation ideal of $\Lambda$ and for a $\Lambda$-module $M$ denote the localization at $I$ by $M_I$. Let $\varrho$ be the length of the torsion submodule of $H^2(C)_I$. If (IMC) holds, then one can show that $\varrho$ coincides with the ``order of vanishing" of $z$, namely, the maximal integer $\varrho$ such that
$$z \in I^\varrho \cdot  \left({\bigwedge}_\Lambda^\sigma H^1(C) \otimes_\Lambda {\bigwedge}_\Lambda^\tau H^2(C)^\ast\right)_I.$$

We can then consider the ``leading term" of $z$. Suppose that a complex $C_0$ of $\ZZ_p$-modules which satisfies the ``control theorem"
$$C\lotimes_\Lambda \ZZ_p  \simeq C_0$$
is given. (Selmer complexes have this property.) 
The leading term 
$$\cD(z) \in \QQ_p\otimes_{\ZZ_p}\left({\bigwedge}_{\ZZ_p}^\sigma H^1(C_0)\otimes_{\ZZ_p} {\bigwedge}_{\ZZ_p}^\tau H^2(C_0)^\ast\right) \otimes_{\ZZ_p} I^{\varrho}/I^{\varrho+1}$$
of $z$ is defined to be the image of $z$ under the ``modulo $I^{\varrho+1}$ map"
$$\cD: I^\varrho \cdot  \left({\bigwedge}_\Lambda^\sigma H^1(C) \otimes_\Lambda {\bigwedge}_\Lambda^\tau H^2(C)^\ast\right)_I \to \QQ_p\otimes_{\ZZ_p}\left({\bigwedge}_{\ZZ_p}^\sigma H^1(C_0)\otimes_{\ZZ_p} {\bigwedge}_{\ZZ_p}^\tau H^2(C_0)^\ast\right) \otimes_{\ZZ_p} I^{\varrho}/I^{\varrho+1}.$$

We are interested in describing the leading term $\cD(z)$. To do this we need several maps. First, we have a canonical ``descent" surjection 
$${\det}_\Lambda^{-1}(C) \twoheadrightarrow {\det}_\Lambda^{-1}(C)\otimes_\Lambda \ZZ_p \simeq {\det}_{\ZZ_p}^{-1}(C\lotimes_\Lambda \ZZ_p) \simeq {\det}_{\ZZ_p}^{-1}(C_0).$$
Next, let $d$ and $e$ be the $\ZZ_p$-ranks of $H^1(C_0)$ and $H^2(C_0)$ respectively. Then we have a canonical isomorphism
$$\pi_0: \QQ_p\otimes_{\ZZ_p} {\det}_{\ZZ_p}^{-1}(C_0) \xrightarrow{\sim} \QQ_p\otimes_{\ZZ_p}\left({\bigwedge}_{\ZZ_p}^d H^1(C_0)\otimes_{\ZZ_p} {\bigwedge}_{\ZZ_p}^e H^2(C_0)^\ast\right).$$
(Here $(-)^\ast$ means the $\ZZ_p$-dual.) Finally, we construct a ``$p$-adic derived Bockstein regulator map"
$$\cR: \QQ_p\otimes_{\ZZ_p}\left({\bigwedge}_{\ZZ_p}^d H^1(C_0)\otimes_{\ZZ_p} {\bigwedge}_{\ZZ_p}^e H^2(C_0)^\ast\right) \to \QQ_p\otimes_{\ZZ_p}\left({\bigwedge}_{\ZZ_p}^\sigma H^1(C_0)\otimes_{\ZZ_p} {\bigwedge}_{\ZZ_p}^\tau H^2(C_0)^\ast\right) \otimes_{\ZZ_p} I^{\varrho}/I^{\varrho+1}.$$
(This is actually injective.) We prove that the following diagram is commutative:
$$
\footnotesize
\xymatrix@C=10pt@R=30pt{
\displaystyle {\det}_\Lambda^{-1}(C) \ar[r]^-{\pi } \ar@{->>}[dd] &\displaystyle I^{\varrho}\cdot \left({\bigwedge}_{\Lambda}^\sigma H^1(C)  \otimes_{\Lambda} {\bigwedge}_{\Lambda}^\tau H^2(C)^\ast \right)_I \ar[rd]^-{\cD} & \\
 & & \displaystyle  \QQ_p\otimes_{\ZZ_p}\left({\bigwedge}_{\ZZ_p}^\sigma H^1(C_0)\otimes_{\ZZ_p} {\bigwedge}_{\ZZ_p}^\tau H^2(C_0)^\ast \right) \otimes_{\ZZ_p}  I^{\varrho}/I^{\varrho +1 } \\
\displaystyle{\det}_{\ZZ_p}^{-1}(C_0) \ar[r]_-{\pi_0} &\displaystyle  \QQ_p\otimes_{\ZZ_p}\left({\bigwedge}_{\ZZ_p}^{d}H^1(C_0) \otimes_{\ZZ_p} {\bigwedge}_{\ZZ_p}^e H^2(C_0)^\ast  \right). \ar[ru]_-{ \cR }&
}
$$
(See Theorem \ref{thm descent}.) If (IMC) holds, then by the commutativity of the diagram we have
$$\text{($p$BSD)} \quad  \ZZ_p\cdot \cD(z) =  \cR\left(  \pi_0 \left({\det}_{\ZZ_p}^{-1}(C_0) \right)\right). $$
This is an abstract formulation of the ``$p$-adic Birch and Swinnerton-Dyer formula up to a $p$-adic unit". In fact, the left hand side corresponds to the leading term of a $p$-adic $L$-function, and the right hand side involves a $p$-adic (derived) regulator. 

In this paper, we explicitly calculate the right hand side of ($p$BSD) in some arithmetic settings. 
Interestingly, if we analyze ($p$BSD) in concrete arithmetic settings, we obtain the exact formulas predicted by Mazur-Tate-Teitelbaum \cite{MTT}, Bertolini-Darmon \cite{BD} and Agboola-Castella \cite{AC}. Thus we are able to obtain results of the type  ``The Iwasawa main conjecture implies the $p$-adic Birch and Swinnerton-Dyer conjecture up to a $p$-adic unit" (see Theorems \ref{main}, \ref{improved AC} and \ref{alg from IMC}). 

We can also study the (complex, not $p$-adic) Birch and Swinnerton-Dyer conjecture via our formalism. More generally, we study the Tamagawa number conjecture of Bloch-Kato \cite{BK}. The basic idea is as follows. Suppose that a special element
$$\eta\in \CC_p\otimes_{\ZZ_p} \left({\bigwedge}_{\ZZ_p}^d H^1(C_0)\otimes_{\ZZ_p} {\bigwedge}_{\ZZ_p}^e H^2(C_0)^\ast\right)$$
is given, and consider the following condition:
$$\text{(TNC)} \quad \pi_0\left({\det}_{\ZZ_p}^{-1}(C_0)\right) = \ZZ_p\cdot \eta.$$
In arithmetic settings, this condition corresponds to the Tamagawa number conjecture. We naturally expect that the special elements $z$ and $\eta$ are related by
$$\text{($p$LTC)} \quad \cD(z) = \cR(\eta), $$
which we call the ``$p$-adic leading term conjecture" (see Conjecture \ref{leading}). Using the commutative diagram above, we can prove the implication
$$\text{(IMC) and ($p$LTC)} \Rightarrow \text{(TNC)}.$$
This is what we study in \S \ref{ks application} (see Theorems \ref{strategy} and \ref{bsd strategy}). Note that the key is the injectivity of the regulator map $\cR$, which is a consequence of our derived construction. (However, determining the order of vanishing $\varrho$ seems to be difficult.)

One of the important topics we do not study in this paper is a connection with the theory of Coleman maps. Let us briefly sketch an idea. Suppose that a morphism of Selmer complexes $C'\to C$ is given. We think of the mapping cone of this morphism as a local cohomology complex. Let $z$ and $z'$ be the special elements for $C$ and $C'$ respectively. Then a Coleman map should give a basis of the determinant of the local cohomology complex, and it induces an isomorphism
$${\rm Col}: {\det}_\Lambda^{-1}(C)\xrightarrow{\sim} {\det}_\Lambda^{-1}(C') ,$$
which sends $z$ to $z'$. By using this Coleman map, one may prove a ``Rubin formula" which relates $\cD(z)$ with $\cD(z')$ (as in \cite[Theorem 6.2]{bks4}). Via the Rubin formula, one may prove the equivalence between ($p$LTC) for $C$ and $C'$ (as in \cite[Corollary 6.7]{bks4}). The author wishes to study this topic in a future work. 

\subsection{Notation}
For any field $K$, we denote the absolute Galois group of $K$ by $G_K$. The algebraic closure of $\QQ$ is denoted by $\overline \QQ$. 

Let $p$ be a prime number and $X$ an abelian group. We set
$$X[p]:=\{x\in X\mid px =0\} \text{ and }X[p^\infty] := \{x \in X \mid p^n x = 0\text{ for some $n$}\}.$$
If $X$ is a $\ZZ_p$-module, then the Pontryagin dual of $X$ is denoted by $X^\vee := \Hom_{\ZZ_p}(X,\QQ_p/\ZZ_p)$. 

Let $R$ be a commutative noetherian ring. For an $R$-module $X$, we set
$$X^\ast:=\Hom_R(X,R).$$

If $R$ is a domain with quotient field $Q$, then for a finitely generated $R$-module $X$ we define the rank of $X$ by
$${\rm rk}_R(X):=\dim_Q(Q\otimes_R X). $$
We denote the $R$-torsion submodule of $X$ by $X_{\rm tors}$ and define the torsion-free quotient by 
$$X_{\rm tf}:=X/X_{\rm tors}.$$

For a finitely generated projective $R$-module $X$, we denote by ${\det}_R(X)$ the determinant module (for the definition, see \cite[Chapter I, \S 3]{kbook} for example). As in \cite[\S 3.2]{bks1}, we regard ${\det}_R(X)$ as a graded invertible $R$-module. We set ${\det}_R^{-1}(X):={\det}_R(X^\ast)$. For a bounded complex $C=[\cdots \to C^i \to C^{i+1}\to \cdots]$ of finitely generated projective $R$-modules, we set
$${\det}_R(C):= \bigotimes_{ i\in \ZZ} {\det}_R^{(-1)^i}(C^i) \text{ and }{\det}_R^{-1}(C):=\bigotimes_{i \in \ZZ}{\det}_R^{(-1)^{i+1}}(C^i).$$
We note that ${\det}_R$ defines a functor from the derived category of perfect complexes of $R$-modules to the category of graded invertible $R$-modules (see \cite{KM}). In particular, we can define ${\det}_R(X)$ for a finitely generated $R$-module $X$ of finite projective dimension.

The following basic properties of the determinant functor are frequently used in this paper. 

\begin{itemize}
\item[(i)] For a finitely generated $R$-module $X$ of finite projective dimension, we have a canonical ``evaluation" isomorphism
$${\rm ev}: {\det}_R(X) \otimes_R {\det}_R^{-1}(X)\xrightarrow{\sim} R.$$
\item[(ii)] For an exact sequence 
$$0\to Y\to X \to Z\to 0$$
of finitely generated $R$-modules of finite projective dimension, we have a canonical isomorphism
$${\det}_R(X) \simeq {\det}_R(Y)\otimes_R {\det}_R(Z).$$
\item[(iii)] Let $C$ be a perfect complex of $R$-modules. If each cohomology $H^i(C)$ has finite projective dimension, then there are canonical isomorphisms
$${\det}_R(C)\simeq \bigotimes_{i \in \ZZ} {\det}_R^{(-1)^i}(H^i(C)) \text { and }{\det}_R^{-1}(C)\simeq \bigotimes_{i \in \ZZ}{\det}_R^{(-1)^{i+1}}(H^i(C)).$$
\item[(iv)] For a commutative $R$-algebra $R'$ and a perfect complex $C$ of $R$-modules, we have a canonical isomorphism
$$R'\otimes_R {\det}_R(C) \simeq {\det}_{R'}(R'\lotimes_R C). $$
\item[(v)] Suppose that $R$ is a principal ideal domain with quotient field $Q$. Then for a finitely generated $R$-module $X$ of rank $r:={\rm rk}_R(X)$ we have
$$\im \left( {\det}_R^{-1}(X) \hookrightarrow Q \otimes_R {\det}_R^{-1}(X) \simeq Q \otimes_R {\bigwedge}_R^{r} X^\ast \right) =  {\rm Fitt}_R(X_{\rm tors})\cdot {\bigwedge}_R^{r} X^\ast,$$
where ${\rm Fitt}_R$ denotes the Fitting ideal over $R$. In particular, when $R=\ZZ_p$, we have a canonical isomorphism
$${\det}_{\ZZ_p}^{-1}(X)\simeq \# X_{\rm tors} \cdot {\bigwedge}_{\ZZ_p}^{r} X^\ast \subset \QQ_p\otimes_{\ZZ_p} {\bigwedge}_{\ZZ_p}^{r}X^\ast.$$
\end{itemize}

\section{The Bockstein formalism}\label{sec derived}

This section is purely algebraic. First, in \S \ref{def bock}, we introduce the notion of ``derived Bockstein maps". Using this, we construct ``derived Bockstein regulators" in \S \ref{sec reg}. Next, in \S \ref{sec structure}, we study the structure of cohomology in terms of derived Bockstein maps. Finally, in \S \ref{sec descent}, we establish a descent formalism involving derived Bockstein regulators (see Theorem \ref{thm descent}). 

\subsection{Derived Bockstein maps}\label{def bock}
Let $p$ be a prime number. 
Let $\cO$ be the ring of integers of a finite extension of $\QQ_p$, and consider the Iwasawa algebra $\Lambda:=\cO[[t]]$. Let
$$I:=\ker (\Lambda \twoheadrightarrow \cO)=(t)$$
be the augmentation ideal. For an integer $k$, we set
$$Q^k:=I^k/I^{k+1}. $$
Note that $Q^k$ is a free $\cO$-module of rank one generated by $t^k$. 
We have a natural identification $Q^k \otimes_\cO Q^{l} =Q^{k+l}$. 

Let $C$ be a bounded complex of finitely generated $\Lambda$-modules. Let $\Lambda_I$ be the localization of $\Lambda$ at the prime ideal $I \subset \Lambda$. We suppose that $C_I:=C\lotimes_\Lambda \Lambda_I$ is acyclic outside degrees one and two and $H^1(C_I)$ is $\Lambda_I$-free. This is equivalent to saying that $C_I$ is represented by
\begin{equation}\label{standard complex}
 \left[ P \to P'\right]
\end{equation}
in the derived category, where $P$ and $P'$ are free $\Lambda_I$-modules of finite rank and $P$ is placed in degree one. We also assume that 
$${\rm rk}_{\Lambda}(H^1(C)) \geq {\rm rk}_{\Lambda}(H^2(C)),$$
i.e., ${\rm rk}_{\Lambda_I}(P)\geq {\rm rk}_{\Lambda_I}(P')$. 
We set 
$$C_0:= C \lotimes_\Lambda \cO .$$ 

Consider the exact sequence of complexes
$$0\to I^{k+1} C\to I^k C \to C_0 \otimes_\cO Q^k \to 0.$$
From this, we obtain a long exact sequence
$$
 H^1(C_0)\otimes_\cO Q^k  \xrightarrow{\delta_k} H^2(I^{k+1}C) \xrightarrow{\iota_k} H^2(I^kC) \xrightarrow{\pi_k} H^2(C_0)\otimes_\cO Q^k .
$$

The usual Bockstein map for $C$ is defined to be the composition map
\begin{equation}\label{beta composition}
\beta = \beta(C): H^1(C_0) \xrightarrow{\delta_0} H^2(IC) \xrightarrow{\pi_1} H^2(C_0)\otimes_\cO Q^1.
\end{equation}
We set $\beta^{(1)}:=\beta$ and define
$$\beta^{(2)}=\beta^{(2)}(C) : \ker \beta^{(1)} \to \coker \beta^{(1)}\otimes_\cO Q^1$$
as follows. Take an element $a \in \ker \beta^{(1)}$. Then, by the definition of $\beta^{(1)}$, we have $\delta_0(a) \in \ker \pi_1 = \im \iota_1$. So there exists $\widetilde a \in H^2(I^2C)$ such that $\delta_0(a) = \iota_1(\widetilde a)$. Then we define $\beta^{(2)}(a) $ to be the image of $ \pi_2(\widetilde a) \in H^2(C_0) \otimes_\cO Q^2$ in $\coker \beta^{(1)} \otimes_\cO Q^1$. One checks that $\beta^{(2)}(a)$ is well-defined: it is independent of the choice of $\widetilde a$. 

In general, $\beta^{(k)}$ is defined as follows. 

\begin{definition}\label{main def}
Let $k\geq 2$. 
We define the {\it $k$-th derived Bockstein map} for $C$
$$\beta^{(k)} = \beta^{(k)}(C): \ker \beta^{(k-1)} \to \coker \beta^{(k-1)}\otimes_\cO Q^1$$
inductively as follows. Take an element $a \in \ker \beta^{(k-1)}$. Then one sees that there exists $\widetilde a \in H^2(I^{k}C )$ such that $\delta_0(a) = \iota_1\circ \cdots \circ \iota_{k-1} (\widetilde a)$.  
We define $\beta^{(k)}(a)$ to be the image of $\pi_{k}(\widetilde a) \in H^2(C_0)\otimes_\cO Q^{k}$ in $\coker \beta^{(k-1)}\otimes_\cO Q^1$. (One checks that this is well-defined.) 
\end{definition}

\begin{remark}\label{spec remark}
The complex $C$ has a natural filtration by $I^k$ and we have a spectral sequence starting with $E_1^{i,j} = H^{i+j}(C_0)\otimes_\cO Q^i $. The definition of the map $\beta^{(k)}$ above is an explicit description of the differential $d_k^{0,1}$ of this spectral sequence. 
\end{remark}

\begin{remark}
The maps $\beta^{(k)}$ define a filtration
\begin{equation}\label{filtration}
H^1(C_0) \supset  \ker \beta^{(1)} \supset \ker \beta^{(2)} \supset \cdots.
\end{equation}
Also, there are natural surjections
\begin{equation}\label{quotient}
H^2(C_0) \otimes_\cO Q^k \twoheadrightarrow \coker \beta^{(1)}\otimes_\cO Q^{k-1} \twoheadrightarrow  \coker \beta^{(2)}\otimes_\cO Q^{k-2} \twoheadrightarrow \cdots .
\end{equation}
\end{remark}

\begin{remark}\label{abstract height}
In arithmetic settings, natural height pairings are constructed by using Bockstein maps in the following way. Suppose that ${\rm rk}_\cO(H^1(C_0)) = {\rm rk}_\cO(H^2(C_0))$ and that a ``cup product"
$$\cup: H^2(C_0)\times H^1(C_0) \to \QQ_p\otimes_{\ZZ_p}\cO$$
is given. (This is usually a perfect pairing.) Then we can define a height pairing 
$$h^{(1)}: H^1(C_0)\times H^1(C_0) \to \QQ_p\otimes_{\ZZ_p} Q^1$$
by
$$h^{(1)}(a, b) := \beta^{(1)}(a)\cup b.$$
More generally, if a cup product
$$\cup_k: \coker \beta^{(k-1)} \times \ker \beta^{(k-1)} \to Q^{k-1}$$
is given, then we can define a $k$-th derived height pairing 
$$h^{(k)}: \ker \beta^{(k-1)}\times \ker \beta^{(k-1)} \to Q^k$$
by
$$h^{(k)}(a,b):= \beta^{(k)}(a) \cup_k b.$$
\end{remark}

\begin{remark}
We remark that Venerucci \cite{venerucci} uses the notion of derived Bockstein maps in a different sense: a morphism in a derived category which induces a Bockstein map is called a derived Bockstein map in loc. cit. So in his sense ``derived" comes from derived categories. On the other hand, our ``derived" comes from ``derived heights" of Bertolini and Darmon \cite{BDMT}, \cite{BD der}. 
\end{remark}

\subsection{Derived Bockstein regulators}\label{sec reg}

We set
$$r:={\rm rk}_\Lambda(H^1(C)) - {\rm rk}_\Lambda(H^2(C)) = {\rm rk}_\cO(H^1(C_0)) - {\rm rk}_{\cO}(H^2(C_0))
$$
and 
$$e:={\rm rk}_\cO(H^2(C_0)).$$
Note that
$$  {\rm rk}_\cO(H^1(C_0))=r +  e.$$
For $k\geq 1$, we set
$$\sigma_k:={\rm rk}_\cO (\ker \beta^{(k)}), \ \tau_k:={\rm rk}_\cO(\coker \beta^{(k)}) \text{ and } e_k :={\rm rk}_\cO(\im \beta^{(k)}).$$

\begin{lemma}\label{elementary}\ 
\begin{itemize}
\item[(i)] We have
$$\sigma_1 \geq \sigma_2 \geq \cdots \geq \sigma_k \geq \cdots $$
and 
$$\tau_1 \geq \tau_2 \geq \cdots \geq \tau_k \geq \cdots.$$
\item[(ii)] For any $k\geq 1$, we have
$$r= \sigma_k-\tau_{k}$$
and
$$e_{k+1}=\sigma_{k}-\sigma_{k+1} =\tau_{k}-\tau_{k+1}.$$
\item[(iii)] For any $k\geq 1$, we have
$$e=e_1+e_2+\cdots + e_{k}+\tau_{k}$$
and
$$e +\sum_{l=1}^{k-1}\tau_l =\sum_{l=1}^{k}l e_l +k\tau_{k}.$$
\item[(iv)] If $\tau_1=0$ (i.e., $\coker \beta^{(1)}$ is torsion), then we have $e_k =0$ for all $k >1$ and $e=e_1$. 
\end{itemize}
\end{lemma}

\begin{proof}
(i) is obvious from (\ref{filtration}) and (\ref{quotient}). (ii) follows from the exact sequence
$$0\to \ker \beta^{(k+1)} \to \ker \beta^{(k)}\xrightarrow{\beta^{(k+1)}}  \coker \beta^{(k)} \otimes_\cO Q^1 \to \coker \beta^{(k+1)} \to 0.$$
(iii) follows by noting
$$\tau_l=e_{l+1}+\tau_{l+1}=e_{l+1}+e_{l+2}+ \cdots +e_{k} +\tau_{k}$$
and
$$e=e_1+\tau_1.$$
(iv) is easily verified. 
\end{proof}

\begin{definition}\label{def der}
Let $k\geq 1$. We set
$$\rho_k:=e+ \sum_{l=1}^{k-1}  \tau_l=\sum_{l=1}^{k}l e_l +k\tau_{k}.$$
We define the {\it $k$-th derived Bockstein regulator isomorphism} for $C$
\begin{multline*}
R^{(k)}=R^{(k)}(C): \QQ_p \otimes_{\ZZ_p}\left( {\bigwedge}_\cO^{r+e} H^1(C_0)\otimes_\cO{\bigwedge}_\cO^e H^2(C_0)^\ast\right)\\
 \xrightarrow{\sim} \QQ_p \otimes_{\ZZ_p}  \left({\bigwedge}_\cO^{\sigma_k} \ker \beta^{(k)} \otimes_\cO  {\bigwedge}_\cO^{\tau_k} \coker \beta^{(k),\ast} \otimes_\cO Q^{\rho_k}\right)
\end{multline*}
as follows. 
First, by the exact sequence
$$ 0 \to \ker \beta^{(1)}\to H^1(C_0) \xrightarrow{\beta^{(1)}} H^2(C_0)\otimes_\cO Q^1 \to \coker \beta^{(1)}\to 0,$$
we obtain a canonical isomorphism
\begin{equation}\label{map1}
{\det}_\cO (H^1(C_0)) \otimes_\cO {\det}_\cO^{-1} (H^2(C_0)) \simeq {\det}_\cO (\ker \beta^{(1)}) \otimes_\cO {\det}_\cO^{-1}(\coker \beta^{(1)}) \otimes_\cO Q^{e}.
\end{equation}
Next, by the exact sequence
$$0 \to \ker \beta^{(2)}\to \ker \beta^{(1)}\xrightarrow{\beta^{(2)}} \coker \beta^{(1)}\otimes_\cO Q^1 \to \coker \beta^{(2)}\to 0,$$
we obtain a canonical isomorphism
\begin{equation}\label{map2}
{\det}_\cO (\ker \beta^{(1)})\otimes_\cO {\det}_\cO^{-1}(\coker \beta^{(1)}) \simeq {\det}_\cO (\ker \beta^{(2)}) \otimes_\cO {\det}_\cO^{-1}(\coker \beta^{(2)}) \otimes_\cO Q^{\tau_1}.
\end{equation}
For $l > 1$, we similarly obtain a canonical isomorphism 
\begin{equation}\label{map3}
{\det}_\cO (\ker \beta^{(l)})\otimes_\cO {\det}_\cO^{-1}(\coker \beta^{(l)}) \simeq {\det}_\cO (\ker \beta^{(l+1)}) \otimes_\cO {\det}_\cO^{-1}(\coker \beta^{(l+1)}) \otimes_\cO Q^{\tau_{l}}.
\end{equation}
Composing (\ref{map1}), (\ref{map2}) and (\ref{map3}) for $l=2,3,\ldots,k-1$, we obtain an isomorphism
$${\det}_\cO(H^1(C_0))\otimes_\cO {\det}_\cO^{-1}(H^2(C_0))  \simeq {\det}_\cO(\ker \beta^{(k)}) \otimes_\cO   {\det}_\cO^{-1}(\coker \beta^{(k)}) \otimes_\cO Q^{\rho_k}.$$
We define $ R^{(k)}$ to be the map induced by this isomorphism. 
\end{definition}

\subsection{Structure of cohomology}\label{sec structure}

We use the notation in the previous subsection. The following number $k_0$ plays an important role. 

\begin{definition}\label{def k0}
Let $k_0\geq 1$ be the minimal integer such that
$$e_{k}:={\rm rk}_\cO(\im \beta^{(k)})=0 \text{ for all $k> k_0$}.$$ 
\end{definition}

\begin{remark}\ 
\begin{itemize}
\item[(i)] We see by Lemma \ref{elementary}(iv) that $k_0=1$ if $\tau_1=0$. 
\item[(ii)] By Lemma \ref{elementary}(ii), we have
$$\sigma_1 \geq \sigma_2\geq \cdots \geq \sigma_{k_0}=\sigma_{k_0+1}=\cdots$$
and 
$$\tau_1 \geq \tau_2\geq \cdots \geq \tau_{k_0}=\tau_{k_0+1}=\cdots.$$
\item[(iii)] The maps $R^{(k)}$ stabilize at $k=k_0$. Namely, if $k \geq k_0$, then $ R^{(k)}$ is identified with $ R^{(k_0)}$. In fact, by the definition of $k_0$, we have
$$\ker \beta^{(k)}=\ker \beta^{(k_0)} \text{ and }\coker \beta^{(k)}=\coker \beta^{(k_0)}\otimes_\cO Q^{k-k_0}$$
for any $k\geq k_0$. Noting that $\sigma_k=\sigma_{k_0}$, $\tau_k=\tau_{k_0}$ and $\rho_k=\rho_{k_0}+ (k-k_0)\tau_{k_0}$, we have an identification
\begin{multline*}
\QQ_p \otimes_{\ZZ_p} \left( {\bigwedge}_\cO^{\sigma_{k_0}} \ker \beta^{(k_0)} \otimes_\cO  {\bigwedge}_\cO^{\tau_{k_0}} \coker \beta^{(k_0),\ast} \otimes_\cO Q^{\rho_{k_0}}\right)\\
 = \QQ_p \otimes_{\ZZ_p} \left( {\bigwedge}_\cO^{\sigma_k} \ker \beta^{(k)} \otimes_\cO  {\bigwedge}_\cO^{\tau_k} \coker \beta^{(k),\ast} \otimes_\cO Q^{\rho_k} \right).
 \end{multline*}
The claim follows from this and the definition of $R^{(k)}$. 
\end{itemize}
\end{remark}

In the following, for a $\Lambda$-module $M$, we denote by $M_I := M \otimes_\Lambda \Lambda_I$ the localization of $M$ at the prime ideal $I \subset \Lambda$. 

\begin{proposition}\label{str h2}
We have an isomorphism
$$H^2(C)_I \simeq (\Lambda_I/I\Lambda_I)^{e_1}\oplus (\Lambda_I /I^2\Lambda_I)^{e_2}\oplus \cdots \oplus (\Lambda_I/I^{k_0}\Lambda_I)^{e_{k_0}} \oplus \Lambda_I^{\tau_{k_0}}.$$
In particular, we have
$${\rm rk}_\Lambda(H^2(C)) = \tau_{k_0},$$
$${\rm length}_{\Lambda_I}(H^2(C)_{I,{\rm tors}})=\sum_{k=1}^{k_0} k e_k$$
and
$$k_0=\min \{ k \mid I^k\cdot H^2(C)_{I,{\rm tors}}=0\}.$$
\end{proposition}

\begin{proof}
We first note that $R:=\Lambda_I$ is a discrete valuation ring with uniformizer $t$. By the structure theorem for finitely generated modules over a discrete valuation ring, we can write
$$H^2(C)_I \simeq (R/(t))^{f_1}\oplus (R/(t^2))^{f_2}\oplus \cdots \oplus (R/(t^l))^{f_l} \oplus R^s.$$
Our aim is to show that $l=k_0$, $f_i = e_i$ and $s=\tau_{k_0}$.

Let $\cQ$ be the quotient field of $\cO$, which is identified with the residue field $R/(t)$. Note that $\cQ \otimes_\cO H^2(C_0)$ is identified with $H^2(C)_I/tH^2(C)_I$ and so
\begin{eqnarray}\label{describe e 2}
e:={\rm rk}_\cO(H^2(C_0)) = {\dim}_\cQ (H^2(C)_I/ t H^2(C)_I) = \sum_{k=1}^l f_k + s.
\end{eqnarray}
Recall that $C_I$ is represented by $[P \xrightarrow{\psi} P']$ (see (\ref{standard complex})). 
We set $d:= {\rm rk}_R(P)$. Note that ${\rm rk}_R(P') = d-r$. 

By choosing suitable bases of $P$ and $P'$, we may assume that $\psi: P\to P'$ is represented by a $(d-r)\times d$ matrix of the form
$$(
{\rm diag}(\underbrace{1,\ldots,1}_{d-r-e}, \overbrace {\underbrace{t, \ldots,t}_{f_1}, \underbrace{t^2,\ldots,t^2}_{f_2},\ldots \ldots,\underbrace{ t^l,\ldots, t^l}_{f_l}, \underbrace{0,\ldots,0}_s}^e)  \mid  O_{d-r,r}),$$
where ${\rm diag}(-)$ denotes the diagonal matrix and $O_{d-r,r}$ denotes the zero matrix of size $(d-r)\times r$. Let $\{b_1,\ldots, b_d\}$ and $\{c_1,\ldots,c_{d-r}\}$ be such bases of $P$ and $P'$ respectively. Then we have
$$\psi(b_i) = \begin{cases}
c_i & \text{ if $1\leq i \leq d-r-e$,}\\
tc_i & \text{ if $d-r-e < i \leq d-r-e+f_1$,}\\
t^2 c_i & \text{ if $d-r-e+f_1< i \leq d-r-e+f_1+f_2$,}\\
\vdots  \\
t^l c_i & \text{ if $d-r-e+ f_1+\cdots + f_{l-1} < i \leq d-r-e+ f_1+\cdots  + f_l$,} \\
0&\text{ if $ d-r-e+f_1+\cdots +f_l < i \leq d$.}
\end{cases}$$

We set $F := P/t P$ and $F' := P'/tP'$. Let $\varphi : F \to F'$ be the map induced by $\psi$. Note that the complex $C_0\lotimes_\cO \cQ$ is identified with $[F \xrightarrow{\varphi} F']$. We denote the images of $b_i \in P$ and $c_i \in P'$ by $\overline b_i \in F$ and $\overline c_i \in F'$ respectively. Then the map $\varphi$ is given by
$$\varphi (\overline b_i) = \begin{cases}
\overline c_i &\text{ if $1\leq i \leq d-r-e$,}\\
0 &\text{ if $d-r-e<i\leq d$.}
\end{cases}$$
So we see that $\cQ\otimes_\cO H^1(C_0) = \ker \varphi$ is generated by $\{\overline b_i \mid d-r-e < i\leq d\}$. Also, $\cQ\otimes_\cO H^2(C_0) = \coker \varphi$ is generated by $\{\overline c_i \mid d-r-e < i \leq d-r\}$. 

We shall now describe the derived Bockstein maps $\beta^{(k)}$ explicitly. Note that the connecting homomorphism $\delta_0$ used in the definition of $\beta=\beta^{(1)}$ (see (\ref{beta composition})) is the ``snake map" associated to the diagram
$$
\xymatrix{
0 \ar[r] & tP \ar[r] \ar[d] & P \ar[r] \ar[d]^\psi & F \ar[r] \ar[d]^\varphi & 0 \\
0 \ar[r]& tP' \ar[r] & P' \ar[r] & F'\ar[r] & 0.
}
$$
So we see that the map $\beta^{(1)}: \cQ\otimes_\cO H^1(C_0) \to \cQ \otimes_\cO H^2(C_0)\otimes_\cO Q^1$ is given by
$$\beta^{(1)}(\overline b_i) = \begin{cases}
\overline c_i \otimes t  & \text{ if $d-r-e< i \leq d-r-e+f_1$,}\\
0 & \text{ if $d-r-e+f_1 < i \leq d$.}
\end{cases}
$$
This shows that $e_1:= {\rm rk}_\cO(\im \beta^{(1)}) = f_1$. We see that $\ker \beta^{(1)}$ is generated by $\{\overline b_i \mid d-r-e+f_1<i \leq d\}$. Similarly, we have
$$\beta^{(2)} (\overline b_i) = \begin{cases}
\overline c_i \otimes t^2 &\text{ if $d-r-e+f_1 < i \leq d-r-e + f_1 + f_2$},\\
0 &\text{ if $d-r-e+f_1+f_2 < i \leq d$}
\end{cases}$$
and $e_2 = f_2$. In general, we have
$$\beta^{(k)}(\overline b_i) = \begin{cases}
\overline c_i \otimes t^k &\text{ if $d-r-e + f_1+\cdots +f_{k-1} < i \leq d-r-e+f_1+\cdots + f_k$,}\\
0 &\text{ if $d-r-e+f_1+\cdots + f_k < i \leq d$}
\end{cases}$$
and $e_k = f_k$. We also see that $\beta^{(k)}$ becomes zero when $k>l$ and so we have $l  = k_0$ by the definition of $k_0$. Since we have
$$s= e-\sum_{k=1}^{k_0}f_k = e-\sum_{k=1}^{k_0} e_k = \tau_{k_0}$$
by (\ref{describe e 2}) and Lemma \ref{elementary}(iii), we have proved the proposition. 
\end{proof}


\begin{remark}
Let $h^{(1)}$ be the height pairing defined in Remark \ref{abstract height} and suppose that the cup product pairing induces an isomorphism $\QQ_p\otimes_{\ZZ_p} H^2(C_0) \simeq \QQ_p\otimes_{\ZZ_p} H^1(C_0)^\ast$. Then, by definition, $h^{(1)}$ is non-degenerate if and only if $\tau_1=0$. Hence, by Proposition \ref{elementary}(iv) and Proposition \ref{str h2}, we see that $h^{(1)}$ is non-degenerate if and only if $H^2(C)$ is $\Lambda$-torsion and
$H^2(C)_I \simeq (\Lambda_I/ I \Lambda_I)^e$. 
\end{remark}

The following proposition claims that the subspace $\QQ_p\otimes_{\ZZ_p} \ker \beta^{(k_0)} \subset \QQ_p\otimes_{\ZZ_p}H^1(C_0)$ coincides with the subspace of ``universal norms". 

\begin{proposition}\label{str h1}
We have a natural isomorphism
$$\QQ_p\otimes_{\ZZ_p} \ker \beta^{(k_0)}  \simeq H^1(C)_I/I  H^1(C)_I.$$
In particular, we have
$${\rm rk}_\Lambda(H^1(C)) = \sigma_{k_0}.$$
\end{proposition}

\begin{proof}
By the proof of Proposition \ref{str h2}, we see that $\cQ \otimes_{\cO}\ker \beta^{(k_0)} (= \QQ_p\otimes_{\ZZ_p}\ker \beta^{(k_0)})$ is generated by $\{\overline b_i \mid d-r-e+e_1+\cdots +e_{k_0} < i \leq d\}$. On the other hand, by the description of the map $\psi$ given there, we see that $H^1(C)_I = \ker \psi$ is generated by $\{b_i \mid  d-r-e+e_1+\cdots + e_{k_0} < i\leq d\}$. This proves the claim. 
\end{proof}

\subsection{The descent formalism}\label{sec descent}


We use the notation in \S \ref{sec reg}. Let $k_0$ be as in Definition \ref{def k0} and set
$$\sigma:=\sigma_{k_0}, \ \tau:=\tau_{k_0} \text{ and }\varrho:=\sum_{k=1}^{k_0} ke_k.$$
Let
\begin{multline*}
R^{(k_0)}: \QQ_p \otimes_{\ZZ_p}\left( {\bigwedge}_\cO^{r+e} H^1(C_0)\otimes_\cO{\bigwedge}_\cO^e H^2(C_0)^\ast\right)\\
 \xrightarrow{\sim} \QQ_p \otimes_{\ZZ_p}  \left({\bigwedge}_\cO^{\sigma} \ker \beta^{(k_0)} \otimes_\cO  {\bigwedge}_\cO^{\tau} \coker \beta^{(k_0),\ast} \otimes_\cO Q^{\rho_{k_0}}\right)
\end{multline*}
be the $k_0$-th derived Bockstein regulator isomorphism in Definition \ref{def der}. Note that  
$$\rho_{k_0} := \sum_{k=1}^{k_0} ke_k + k_0 \tau = \varrho + k_0\tau.$$
Since we have a natural surjection
$$H^2(C_0)\otimes_\cO Q^{k_0} \twoheadrightarrow \coker \beta^{(k_0)}$$
(see (\ref{quotient})), there is a natural injection
$$\QQ_p\otimes_{\ZZ_p}{\bigwedge}_\cO^\tau \coker \beta^{(k_0),\ast} \otimes_\cO Q^{\rho_{k_0}} \hookrightarrow \QQ_p\otimes_{\ZZ_p}{\bigwedge}_\cO^\tau H^2(C_0)^\ast \otimes_\cO Q^\varrho.$$
Since we have $\ker \beta^{(k_0)} \subset H^1(C_0)$, we can regard $R^{(k_0)}$ as a map
\begin{multline} \label{def der map}
R^{(k_0)}: \QQ_p\otimes_{\ZZ_p} \left( {\bigwedge}_\cO^{r+e} H^1(C_0)\otimes_\cO {\bigwedge}_{\cO}^e H^2(C_0)^\ast\right) \\
 \to \QQ_p\otimes_{\ZZ_p} \left({\bigwedge}_{\cO}^\sigma H^1(C_0) \otimes_\cO {\bigwedge}_\cO^{\tau} H^2(C_0)^\ast \otimes_\cO  Q^{\varrho}\right).
\end{multline}

\begin{theorem}\label{thm descent}
We have a commutative diagram:
$$
\small
\xymatrix@C=10pt@R=30pt{
\displaystyle {\det}_\Lambda^{-1}(C) \ar[r]^-{\pi } \ar@{->>}[dd] &\displaystyle I^{\varrho}\cdot {\bigwedge}_{\Lambda_I}^\sigma H^1(C)_I  \otimes_{\Lambda_I} {\bigwedge}_{\Lambda_I}^\tau H^2(C)_I^\ast \ar[rd]^-{\cD} & \\
 & & \displaystyle  \QQ_p\otimes_{\ZZ_p}\left({\bigwedge}_\cO^\sigma H^1(C_0)\otimes_\cO {\bigwedge}_\cO^\tau H^2(C_0)^\ast \otimes_\cO Q^{\varrho}\right)  \\
\displaystyle{\det}_\cO^{-1}(C_0) \ar[r]_-{\pi_0} &\displaystyle  \QQ_p\otimes_{\ZZ_p}\left({\bigwedge}_\cO^{r+e}H^1(C_0) \otimes_\cO {\bigwedge}_\cO^e H^2(C_0)^\ast  \right) \ar[ru]_-{ R^{(k_0)}},&
}
$$
where the left vertical map is the natural surjection and the maps $\pi$, $\cD$ and $\pi_0$ are defined below. 
\end{theorem}

\begin{remark}
When $k_0=1$, Theorem \ref{thm descent} is essentially proved in \cite[Lemma 5.22]{bks1} and used in various settings (such as \cite[Lemma 5.17]{bks2}, \cite[Theorem 7.8]{bks4} and \cite[Proposition 5.4]{bks5}). Theorem \ref{thm descent} is a generalization of this result for derived Bockstein maps. 
\end{remark}

\begin{remark}\label{more general}
When $\sigma = \tau =0$ (i.e., $H^1(C)$ and $H^2(C)$ are $\Lambda$-torsion), Theorem \ref{thm descent} is essentially proved by Nekov\'a\v r in \cite[(11.6.12.2)]{nekovar}. 
\end{remark}

We shall define the maps appearing in the diagram in Theorem \ref{thm descent}. 

First, $\pi$ is defined as follows. Note that, by Propositions \ref{str h2} and \ref{str h1}, we have
$$\sigma = {\rm rk}_\Lambda(H^1(C)), \ \tau = {\rm rk}_\Lambda (H^2(C)) \text{ and } {\rm Fitt}_{\Lambda_I}(H^2(C)_{I, {\rm tors}}) = I^\varrho \Lambda_I.$$
We define $\pi$ to be map induced by the canonical isomorphism
$${\det}_\Lambda^{-1}(C)_I \simeq  {\det}_{\Lambda_I}(H^1(C)_I )\otimes_{\Lambda_I} {\det}_{\Lambda_I}^{-1} (H^2(C)_I) = I^{\varrho} \cdot {\bigwedge}_{\Lambda_I}^\sigma H^1(C)_I \otimes_{\Lambda_I} {\bigwedge}_{\Lambda_I}^\tau H^2(C)_I^\ast.$$

Next, we define $\cD$. For $x \in {\bigwedge}_{\Lambda_I}^\sigma H^1(C)_I$, we denote the image of $x$ under the natural map
$${\bigwedge}_{\Lambda_I}^\sigma H^1(C)_I \twoheadrightarrow \left( {\bigwedge}_{\Lambda_I}^\sigma H^1(C)_I\right) \otimes_{\Lambda_I} \Lambda_I/ I\Lambda_I \simeq \QQ_p \otimes_{\ZZ_p} {\bigwedge}_\cO^\sigma \ker \beta^{(k_0)} \subset \QQ_p\otimes_{\ZZ_p} {\bigwedge}_\cO^\sigma H^1(C_0)$$
by $\overline x$. (Here we used Proposition \ref{str h1} for the last isomorphism.) Similarly, for $y \in {\bigwedge}_{\Lambda_I}^\tau H^2(C)_I^\ast$, we denote its natural image in $\QQ_p\otimes_{\ZZ_p} {\bigwedge}_\cO^\tau H^2(C_0)^\ast$ by $\overline y$. Also, the image of $a \in I^{\varrho}$ in $Q^{\varrho} = I^{\varrho}/I^{\varrho+1}$ is denoted by $\overline a$. Now we define $\cD$ by
$$\cD: I^{\varrho}\cdot {\bigwedge}_{\Lambda_I}^\sigma H^1(C)_I  \otimes_{\Lambda_I} {\bigwedge}_{\Lambda_I}^\tau H^2(C)_I^\ast \to  \QQ_p\otimes_{\ZZ_p}\left({\bigwedge}_\cO^\sigma H^1(C_0)\otimes_\cO {\bigwedge}_\cO^\tau H^2(C_0)^\ast \otimes_\cO Q^{\varrho}\right)$$
$$ a \cdot x\otimes y \mapsto \overline x \otimes \overline y \otimes \overline a.$$

Finally, $\pi_0$ is defined to be the map induced by the canonical  isomorphism
\begin{multline*}
\pi_0 : \QQ_p\otimes_{\ZZ_p}{\det}_\cO^{-1}(C_0) \simeq \QQ_p\otimes_{\ZZ_p} \left( {\det}_\cO(H^1(C_0)) \otimes_\cO {\det}_\cO^{-1}(H^2(C_0))\right) \\
=\QQ_p\otimes_{\ZZ_p}\left( {\bigwedge}_\cO^{r+e}H^1(C_0) \otimes_\cO {\bigwedge}_\cO^e H^2(C_0)^\ast\right).
\end{multline*}

In the proof of Theorem \ref{thm descent}, we will use the following lemma, which will be useful when we describe $\pi$, $\pi_0$ and $R^{(k_0)}$ explicitly.

\begin{lemma}\label{det lemma}
Let $L$ be a field and 
$$0 \to W \to U\xrightarrow{\varphi} V \to X\to 0$$
an exact sequence of finite dimensional $L$-vector spaces. We set
$$a := \dim_L(U), \ b:= \dim_L(V) \text{ and } c:= \dim_L(W). $$
(Note that $\dim_L(X)=b-a+c$.) Let $\{ u_1,\ldots,u_a\}$ and $\{v_1,\ldots,v_b\}$ be bases of $U$ and $V$ respectively. We assume that $\varphi$ is given by
\begin{equation}\label{phi describe}
\varphi(u_i) = \begin{cases}
\alpha_i v_i & \text{ if $1\leq i \leq a-c$,}\\
0 &\text{ if $a-c< i \leq a$}
\end{cases}
\end{equation}
for some $\alpha_i \in L^\times$. We identify $W$ with the subspace of $U$ generated by $\{u_i \mid a-c <i\leq a\}$. Let $\{v_1^\ast,\ldots,v_b^\ast\}$ be the dual basis of $\{v_1,\ldots,v_b\}$. We identify $X^\ast$ with the subspace of $V^\ast$ generated by $\{v_i^\ast \mid a-c < i \leq b\}$. 

Then the canonical isomorphism
\begin{equation}\label{uvisom}
{\bigwedge}_L^a U \otimes_L {\bigwedge}_L^b V^\ast = {\det}_L(U)\otimes_L {\det}_L^{-1}(V) \xrightarrow{\sim} {\det}_L(W)\otimes_L {\det}_L^{-1}(X) = {\bigwedge}_L^c W \otimes_L {\bigwedge}_L^{b-a+c} X^\ast
\end{equation}
is explicitly given by
$$u_1\wedge\cdots \wedge u_a \otimes v_1^\ast \wedge\cdots \wedge v_b^\ast \mapsto \alpha_1\cdots \alpha_{a-c}\cdot u_{a-c+1}\wedge\cdots \wedge u_a \otimes v_{a-c+1}^\ast \wedge\cdots \wedge v_b^\ast.$$
\end{lemma}

\begin{proof}
This is well-known and we omit the proof. 
\end{proof}

\begin{remark}
More generally, if we assume
$$\varphi(u_i) = \begin{cases}
\displaystyle \sum_{j=1}^{a-c} \alpha_{ij} v_j &\text{ if $1\leq i \leq a-c$,}\\
0 &\text{ if $a-c < i \leq a$}
\end{cases}$$
(for some $\alpha_{ij} \in L$) instead of (\ref{phi describe}), then the isomorphism (\ref{uvisom}) is given by
$$u_1 \wedge \cdots \wedge u_a\otimes v_1^\ast \wedge \cdots \wedge v_b^\ast \mapsto  \det(\alpha_{ij})_{1\leq i,j\leq a-c} \cdot u_{a-c+1}\wedge\cdots \wedge u_a \otimes v_{a-c+1}^\ast \wedge\cdots \wedge v_b^\ast.$$
\end{remark}

We now give a proof of Theorem \ref{thm descent}. 

\begin{proof}[Proof of Theorem \ref{thm descent}]
We set $R:=\Lambda_I$. Let $\cQ$ be the quotient field of $\cO$, which is identified with the residue field of $R$. Recall from (\ref{standard complex}) that $C_I$ is represented by $[P \xrightarrow{\psi} P']$ and set $d:={\rm rk}_R(P)$. Recall from the proof of Proposition \ref{str h2} that we can choose bases $\{b_1,\ldots,b_d\}$ and $\{c_1,\ldots,c_{d-r}\}$ of $P$ and $P'$ respectively so that $\psi$ is given by
$$\psi(b_i) = \begin{cases}
c_i & \text{ if $1\leq i \leq d-r-e$,}\\
tc_i & \text{ if $d-r-e < i \leq d-r-e+e_1$,}\\
t^2 c_i & \text{ if $d-r-e+e_1< i \leq d-r-e+e_1+e_2$,}\\
\vdots  \\
t^{k_0} c_i & \text{ if $d-r-e+ e_1+\cdots + e_{k_0-1} < i \leq d-r-e+ e_1+\cdots  + e_{k_0}$,} \\
0&\text{ if $ d-r-e+e_1+\cdots +e_{k_0} < i \leq d$.}
\end{cases}$$
Note that $\sigma = r+\tau = r+ e-e_1-\cdots -e_{k_0}$ (see Lemma \ref{elementary}(iii)). Letting $\fQ$ be the quotient field of $R$ and applying Lemma \ref{det lemma} to the exact sequence
$$0\to \fQ\otimes_R H^1(C)_I \to \fQ\otimes_R P \xrightarrow{\psi} \fQ\otimes_R P'\to \fQ\otimes_R H^2(C)_I \to 0,$$
we see that the map
$$\pi: {\det}_\Lambda^{-1}(C)_I = {\bigwedge}_R^d P \otimes_R {\bigwedge}_R^{d-r} (P')^\ast \to I^{\varrho}\cdot {\bigwedge}_{R}^\sigma H^1(C)_I  \otimes_{R} {\bigwedge}_{R}^\tau H^2(C)_I^\ast$$
is explicitly given by
$$b_1\wedge\cdots \wedge b_d \otimes c_1^\ast\wedge\cdots \wedge c_{d-r}^\ast \mapsto t^{\varrho}\cdot b_{d-\sigma+1}\wedge\cdots \wedge b_d \otimes c_{d-\sigma+1}^\ast \wedge\cdots \wedge c_{d-r}^\ast.$$
So we have
$$(\cD \circ \pi) (b_1\wedge\cdots \wedge b_d \otimes c_1^\ast\wedge\cdots \wedge c_{d-r}^\ast)= \overline {b_{d-\sigma+1}\wedge\cdots \wedge b_d} \otimes \overline {c_{d-\sigma+1}^\ast \wedge\cdots \wedge c_{d-r}^\ast}\otimes t^\varrho.$$

As in the proof of Proposition \ref{str h2}, we set $F:=P/tP$ and $F':=P'/tP'$ and let $\varphi: F\to F'$ be the map induced by $\psi$. Let $\overline b_i \in F$ and $\overline c_i \in F'$ be the images of $b_i \in P$ and $c_i \in P'$ respectively. Then, since $\varphi$ is given by
$$\varphi (\overline b_i) = \begin{cases}
\overline c_i &\text{ if $1\leq i \leq d-r-e$,}\\
0 &\text{ if $d-r-e<i\leq d$,}
\end{cases}$$
Lemma \ref{det lemma} implies that the map
$$\pi_0: \QQ_p \otimes_{\ZZ_p} {\det}_\cO^{-1}(C_0) = {\bigwedge}_\cQ^d F \otimes_\cQ {\bigwedge}_\cQ^{d-r} (F')^\ast \to \QQ_p\otimes_{\ZZ_p}\left({\bigwedge}_\cO^{r+e}H^1(C_0) \otimes_\cO {\bigwedge}_\cO^e H^2(C_0)^\ast  \right) $$
is explicitly given by
$$\overline b_1\wedge\cdots \wedge \overline b_d \otimes \overline c_1^\ast \wedge\cdots \wedge \overline c_{d-r}^\ast \mapsto  \overline b_{d-r-e+1}\wedge \cdots \wedge \overline b_d \otimes \overline c_{d-r-e+1}^\ast \wedge\cdots \wedge \overline c_{d-r}^\ast.$$

To complete the proof, it is now sufficient to show that the map $R^{(k_0)}$ is given by
$$\overline b_{d-r-e+1}\wedge \cdots \wedge \overline b_d \otimes \overline c_{d-r-e+1}^\ast \wedge\cdots \wedge \overline c_{d-r}^\ast \mapsto \overline b_{d-\sigma+1}\wedge\cdots \wedge \overline b_d \otimes \overline c_{d-\sigma+1}^\ast \wedge\cdots \wedge \overline c_{d-r}^\ast\otimes t^\varrho.$$
But this is proved by applying Lemma \ref{det lemma} again, since each $\beta^{(k)}$ is explicitly given by
$$\beta^{(k)}(\overline b_i) = \begin{cases}
\overline c_i \otimes t^k &\text{ if $d-r-e + e_1+\cdots +e_{k-1} < i \leq d-r-e+e_1+\cdots + e_k$,}\\
0 &\text{ if $d-r-e+e_1+\cdots + e_k < i \leq d$.}
\end{cases}$$
(See the proof of Proposition \ref{str h2}.)
\end{proof}

\section{A conjecture of Bertolini and Darmon}\label{first application}

In this section, as an application of the descent formalism in the previous section, we prove that a conjecture of Birch and Swinnerton-Dyer type for Heegner points formulated by Bertolini and Darmon in \cite{BD} follows from the Heegner point main conjecture up to a $p$-adic unit (see Theorem \ref{main}). 

\subsection{Notation and hypotheses}\label{sec bd}
Let $p$ be an odd prime number and $K$ an imaginary quadratic field in which $p$ does not ramify. Let $E$ be an elliptic curve defined over $\QQ$ with conductor $N$. We often regard $E$ as an elliptic curve over $K$. We denote the Tate-Shafarevich group for $E/K$ by $\sha(E/K)$. Throughout this section, we assume the following. 
\begin{hypothesis}\label{heegner hyp}\ 
\begin{itemize}
\item[(i)] (Heegner hypothesis) every prime divisor of $N$ splits in $K$,
\item[(ii)] $E$ has good ordinary reduction at $p$,
\item[(iii)] $E(K)[p]=0$,
\item[(iv)] $\sha(E/K)[p^\infty]$ is finite. 
\end{itemize}
\end{hypothesis}
Let $T:=T_p(E)$ be the $p$-adic Tate module of $E$. We set
$$V:= \QQ_p\otimes_{\ZZ_p} T \text{ and }A:=V/T.$$
Let $K_\infty/K$ be the anticyclotomic $\ZZ_p$-extension and $K_n$ its $n$-th layer. We set
$$\Gamma_n:=\Gal(K_n/K), \ \Gamma:=\Gal(K_\infty/K) \text{ and }\Lambda:=\ZZ_p[[\Gamma]].$$
Let
$$\TT:=T\otimes_{\ZZ_p}\Lambda $$
be the induced module, on which $G_K$ acts by 
$$\sigma\cdot (t \otimes\lambda) := \sigma t \otimes \overline \sigma^{-1}\lambda \quad (\sigma \in G_K, \ t \in T, \ \lambda \in \Lambda),$$
where $\overline \sigma^{-1} \in \Gamma$ denotes the image of $\sigma^{-1} \in G_K$. $\TT$ is denoted by $\mathscr{F}_\Gamma(T) $ in \cite[(8.3.1)]{nekovar} (see also \cite[Proposition 8.4.4.1]{nekovar}). Note that $\TT$ is identified with the inverse limit $\varprojlim_n {\rm Ind}_{K_n/K}(T)$. We also consider the direct limit
$$\bA:= \varinjlim_n {\rm Ind}_{K_n/K}(A),$$
which is denoted by $F_\Gamma(A)$ in \cite[(8.3.1)]{nekovar}. 

\subsection{Selmer complexes and height pairings}\label{review height}

Let $S$ be the set of places of $K$ which divide $pN \infty$. 
Let $X \in \{T,V,A,\TT, \bA\}$. Since $E$ has good ordinary reduction at $v \mid p$, we have a canonical $G_{K_v}$-submodule $X^+ =X^+_v\subset X$ (see \cite[(9.6.7.2)]{nekovar} for example). We set $X^-:=X/X^+$. For $v\mid N$, we denote the maximal unramified extension of $K_v$ by $K_v^{\rm ur}$ and set 
$$\rgamma_{\rm ur}(K_v, X):=\rgamma(K_v^{\rm ur}/K_v, X^{G_{K_v^{\rm ur}}}) \text{ and } \rgamma_{/\rm ur}(K_v,X):={\rm Cone}\left(\rgamma_{\rm ur}(K_v,X) \to \rgamma(K_v,X) \right).$$
We define a Selmer complex $\widetilde \rgamma_f(K,X)$ by the exact triangle
\begin{equation}\label{triangle}
\widetilde \rgamma_f(K,X)\to \rgamma(\cO_{K,S},X) \to \bigoplus_{v \mid p}\rgamma(K_v,X^-) \oplus \bigoplus_{v \mid N} \rgamma_{/\rm ur}(K_v,X).
\end{equation}
(See \cite[(6.1.3.2)]{nekovar}.) Here we set $\rgamma(\cO_{K,S}, X):=\rgamma_{\rm cont}(G_{K,S},X)$ (see \cite[(5.1)]{nekovar}). We denote $H^i(\widetilde \rgamma_f(K,X))$ by $\widetilde H^i_f(K,X)$. 


Let $I$ be the augmentation ideal of $\Lambda$. 
By \cite[Proposition 9.7.6(iii)]{nekovar}, we know that $\widetilde \rgamma_f(K,\TT) \lotimes_\Lambda \Lambda_I$ is represented by a complex of the form (\ref{standard complex}) (with $P=P'$). Also, we know that the ``control theorem"
$$\widetilde \rgamma_f(K,\TT) \lotimes_\Lambda \ZZ_p \simeq \widetilde \rgamma_f(K,T)$$
holds (see \cite[Proposition 8.10.1]{nekovar}). 


Hence, we can apply the construction in \S \ref{sec derived} with $C= \widetilde \rgamma_f(K,\TT)$. Recall that we set $Q^k:=I^k/I^{k+1}$. 
Let $\widetilde \beta^{(k)}:=\beta^{(k)}(\widetilde \rgamma_f(K,\TT))$ be the derived Bockstein map (see Definition \ref{main def}):
$$\widetilde\beta^{(1)}: \widetilde H^1_f(K,T) \to \widetilde H^2_f(K,T) \otimes_{\ZZ_p} Q^1,$$
$$\widetilde\beta^{(k)}: \ker \widetilde\beta^{(k-1)}\to \coker \widetilde\beta^{(k-1)}\otimes_{\ZZ_p} Q^1.$$

By our assumption that $\sha(E/K)[p^\infty]$ is finite, we have a natural identification
$$\widetilde H_f^1(K,V) = \QQ_p \otimes_\ZZ E(K). $$
(See \cite[Lemma 9.6.3 and Lemma 9.6.7.3(iv)]{nekovar}.) Moreover, the Kummer dual $V^\ast(1)$ of $V$ is identified with $V$ via the Weil pairing and so 
$$\widetilde H^2_f(K,V) \simeq \widetilde H^1_f(K,V)^\ast$$
by duality (see \cite[Proposition 9.7.3(iii)]{nekovar}). Therefore, the Bockstein map $\widetilde \beta^{(1)}$ induces a pairing
$$\widetilde h^{(1)}: E(K)\times E(K) \to \QQ_p\otimes_{\ZZ_p}Q^1. $$
This coincides (up to sign) with the classical $p$-adic height pairing (see \cite[Theorem 11.3.9]{nekovar}). 

As noted in Remark \ref{spec remark}, $\widetilde \beta^{(k)}$ coincides with the differential $d_k^{0,1}$ of the spectral sequence starting with $E_1^{i,j}= \widetilde H^{i+j}_f(K,T)\otimes_{\ZZ_p} Q^i $. 
So we have 
$$\ker \widetilde \beta^{(k-1)} = E_k^{0,1} \text{ and }\coker \widetilde \beta^{(k-1)}\otimes_{\ZZ_p}Q^1 = E_k^{k,2-k}.$$
Let
$$\cup_k: E_k^{k,2-k}\times E_k^{0,1} \to Q^k$$
be the cup product defined in \cite[(11.5.3)]{nekovar}.
Following \cite[(11.5.5)]{nekovar}, we define a $k$-th derived height pairing
$$\widetilde h^{(k)}: \ker \widetilde \beta^{(k-1)} \times \ker \widetilde \beta^{(k-1)} \to Q^k$$
by
$$\widetilde h^{(k)}(a,b) := \widetilde \beta^{(k)}(a)\cup_k b.$$
This pairing has the following properties. 

\begin{proposition}[Nekov\'a\v r]\label{higher height}\ 
\item[(i)] $\widetilde h^{(k)}$ is symmetric (resp. alternating) if $k$ is odd (resp. even). 
\item[(ii)] Let $\tau \in \Gal(K/\QQ)$ be the complex conjugation. Then we have
$$\widetilde h^{(k)}(\tau a, \tau b)= (-1)^k \widetilde h^{(k)}(a,b). $$
\item[(iii)] $\widetilde h^{(k)}$ coincides (up to sign) with the ``derived $p$-adic height" defined by Howard in \cite{howard der}. 
\end{proposition}

\begin{proof}
See \cite[Proposition 11.5.6(ii)]{nekovar},  \cite[Proposition 11.5.9(iii)]{nekovar} and \cite[Proposition 11.8.4]{nekovar}.
\end{proof}

\subsection{Formulation of the conjecture}\label{formulation bd}




We formulate a slightly modified version of the conjecture of Bertolini and Darmon \cite[Conjecture 4.5(1)]{BD}. 

We first note that the $\Lambda$-rank of $\widetilde H^1_f(K,\TT)$ is known to be one (see \cite[Lemma 2.3]{nekovar parity} and \cite[Theorem 5.5.2]{CGS}). Since the Euler-Poincar\'e characteristic of $\widetilde \rgamma_f(K,\TT)$ is zero (see \cite[Theorem 7.8.6]{nekovar}), we know also that ${\rm rk}_\Lambda(\widetilde H^2_f(K,\TT)) = 1$. 

We use the notation in \S \ref{sec reg}. In the present setting with $C=\widetilde \rgamma_f(K,\TT)$, note that we have
$$r=0, \ \sigma_k=\tau_k = {\rm rk}_{\ZZ_p}(\ker \widetilde \beta^{(k)}) \text{ and } e={\rm rk}_\ZZ(E(K)).$$
We set $e_k:={\rm rk}_{\ZZ_p}(\im \widetilde \beta^{(k)}) = \sigma_{k-1}-\sigma_k$. Recall from Definition \ref{def k0} that $k_0 \geq 1$ is the minimal integer such that $e_k=0$ for all $k> k_0$. Since ${\rm rk}_\Lambda(\widetilde H^1_f(K,\TT))=1$, we have $\sigma_{k_0}=\tau_{k_0}=1$ (see Propositions \ref{str h2} and \ref{str h1}). So the $k_0$-th derived Bockstein regulator map in (\ref{def der map}) becomes a map
$$\widetilde R^{(k_0)} := R^{(k_0)}(\widetilde \rgamma_f(K,\TT)): \QQ_p \otimes_\ZZ \left( {\bigwedge}_\ZZ^e E(K) \otimes_\ZZ {\bigwedge}_\ZZ^e E(K)\right)\to \QQ_p\otimes_{\ZZ_p} \left( E(K) \otimes_\ZZ E(K)  \otimes_{\ZZ} Q^{\varrho}\right)$$
with $\varrho:=\sum_{k=1}^{k_0} ke_k$.

\begin{remark}\label{rem order}
We set 
$$r^+ := {\rm rk}_\ZZ (E(\QQ)), \ r^- := e-r^+ \text{ and }\nu:=2(\max\{r^+,r^-\}-1).$$
Then we have
$$\varrho \geq \nu.$$
In fact, since $\tau_{k_0}=1$, we know by Lemma \ref{elementary}(iii) that
$$\varrho  = e+ \sum_{k=1}^{k_0-1}\tau_k - k_0=e-1 + \sum_{k=1}^{k_0-1} (\tau_{k} -1).$$
By Proposition \ref{higher height}(ii), we have $\tau_1 \geq |r^+-r^-|$ (see \cite[Remark 3.5]{AC}). Also, we have $\tau_k \geq \tau_{k_0}=1$ for any $k\geq 2$. Hence we have
$$\varrho \geq e-1 + |r^+-r^-| -1 = \nu.$$
\end{remark}

\begin{remark}\label{MBD}
According to a conjecture of Mazur and Bertolini-Darmon (see \cite[Conjecture 3.8]{AC} for example), we should have
$$\tau_1 = |r^+-r^-| \text{ and } \tau_k=1 \text{ for $k\geq2$}$$
and so
$$e_1= e-|r^+-r^-| (= 2 \min\{r^+,r^-\}), \ e_2 = |r^+-r^-|-1 \text{ and } e_k = 0 \text{ for $k>2$}.$$
So conjecturally we have
$$k_0 = \begin{cases}
1 & \text{ if $|r^+-r^-|=1$,}\\
2 & \text{ if $|r^+-r^-|>1$}
\end{cases}
$$
and
$$
\varrho= e_1 + 2e_2 = \nu.
$$
\end{remark}

\begin{definition}\label{bd reg}
Let $\{x_1,\ldots,x_e\}$ be a $\ZZ$-basis of $E(K)_{\rm tf}$. We define the {\it Bertolini-Darmon derived regulator}
$${\rm Reg}^{\rm BD} \in \QQ_p\otimes_{\ZZ_p}(E(K)\otimes_\ZZ E(K)\otimes_\ZZ Q^\nu)$$
by
$${\rm Reg}^{\rm BD}:= \begin{cases}
{\widetilde R}^{(k_0)}((x_1\wedge\cdots \wedge x_e)\otimes (x_1\wedge\cdots \wedge x_e))  &\text{ if $\varrho = \nu$},\\
0 &\text{ if $\varrho > \nu$.}
\end{cases}$$
(This is obviously independent of the choice of $\{x_1,\ldots,x_e\}$.)
\end{definition}


Fix a modular parametrization $\phi: X_0(N)\to E$ and let $y_K \in E(K)$ be the Heegner point associated to $\phi$. Let $c_\phi$ be the Manin constant and set $u_K:=\# \cO_K^\times /2$. Let $\alpha \in \ZZ_p^\times$ be the unit root of the Hecke polynomial $X^2-a_p X +p$.  
Let
$$z_\infty=(z_n)_n \in \varprojlim_n (\ZZ_p\otimes_\ZZ E(K_n))$$
be the system of regularized Heegner points (see \cite[\S 2.5]{BD}). We note that $z_0$ is related with $y_K$ by
$$z_0 = \frac{1}{u_K}L_p \cdot y_K,$$
where $L_p \in \ZZ_p$ is defined by
\begin{equation}\label{lpdef}
L_p =\begin{cases}
1-\alpha^{-2} & \text{ if $p$ is inert in $K$},\\
(1-\alpha^{-1})^2 & \text{ if $p$ is split in $K$.}
\end{cases}
\end{equation}

We set
$$\theta_n:=\sum_{\sigma \in \Gamma_n} \sigma z_n \otimes \sigma^{-1} \in (\ZZ_p\otimes_\ZZ E(K_\infty)) \otimes_{\ZZ_p} \ZZ_p[\Gamma_n].$$
Then one sees that the map induced by the natural surjection $\ZZ_p[\Gamma_{n+1}]\twoheadrightarrow \ZZ_p[\Gamma_n]$ sends $\theta_{n+1}$ to $\theta_n$ and we can define
$$\theta := (\theta_n)_n \in \varprojlim_n (\ZZ_p\otimes_\ZZ E(K_\infty)) \otimes_{\ZZ_p} \ZZ_p[\Gamma_n] = E(K_\infty)\otimes_{\ZZ}\Lambda.$$
Let $\theta^\iota$ be the image of $\theta$ under the map induced by the involution $\iota: \Gamma \to \Gamma; \ \gamma \mapsto \gamma^{-1}$. 
We define
$$\cL:=\theta\otimes \theta^\iota \in(E(K_\infty)\otimes_{\ZZ}\Lambda) \otimes_\Lambda (E(K_\infty)\otimes_{\ZZ}\Lambda)\simeq  E(K_\infty)\otimes_\ZZ E(K_\infty) \otimes_\ZZ \Lambda$$

We now state the conjecture of Bertolini and Darmon. We denote the product of Tamagawa factors of $E/K$ by ${\rm Tam}(E/K)$. 

\begin{conjecture}\label{bd}
Set $\nu := 2 (\max\{r^+,r^-\}-1)$. 
\begin{itemize}
\item[(i)] We have $\cL \in E(K_\infty)\otimes_\ZZ E(K_\infty)\otimes_\ZZ I^\nu$.
\item[(ii)] The image $\overline \cL$ of $\cL$ in $E(K_\infty)\otimes_\ZZ E(K_\infty)\otimes_\ZZ Q^\nu$ is contained in $E(K)\otimes_\ZZ E(K)\otimes_\ZZ Q^\nu$ and we have
$$\overline \cL = (c_\phi L_p)^2 \cdot \frac{\# \sha(E/K)\cdot {\rm Tam}(E/K) }{\#E(K)_{\rm tors}^2}\cdot {\rm Reg}^{\rm BD} $$
in $ \QQ_p\otimes_{\ZZ_p}\left( E(K)\otimes_\ZZ E(K) \otimes_\ZZ Q^\nu\right)$. (Here $\sha(E/K)$ is assumed to be finite.)
\end{itemize}
\end{conjecture}

\begin{remark}
Conjecture \ref{bd} is a slight modification of a derived version of the conjecture of Bertolini-Darmon \cite[Conjecture 4.5(1)]{BD} (see also \cite[Conjecture 3.11]{AC}, where $(1-a_p(E)+p)/p$ should be replaced by $L_p$). 
\end{remark}

\begin{remark}
If Conjecture \ref{bd}(i) is true, then the first claim in Conjecture \ref{bd}(ii) (that $\overline \cL$ lies in $E(K)\otimes_\ZZ E(K)\otimes_\ZZ Q^\nu$) is also true (see \cite[Lemma 3.3]{AC}). 
\end{remark}

\begin{remark}\label{BD rank one}
When the analytic rank of $E/K$ is one, Conjecture \ref{bd} is equivalent to the Birch and Swinnerton-Dyer conjecture. In fact, in this case we have ${\rm rk}_\ZZ(E(K))=1$ (by the theorem of Gross-Zagier-Kolyvagin) and by Proposition \ref{higher height}(ii) we see that the map $\widetilde \beta^{(1)}$ is zero. 
So we have $\varrho=\nu = 0$ and Conjecture \ref{bd}(i) is trivially true. Also, we have ${\rm Reg}^{\rm BD} = x\otimes x$ by definition, where $x$ is a basis of $E(K)_{\rm tf}$. Hence Conjecture \ref{bd}(ii) is equivalent to
$$z_0\otimes z_0 = (c_\phi L_p)^2 \cdot \frac{\# \sha(E/K) \cdot {\rm Tam}(E/K)}{\# E(K)_{\rm tors}^2}\cdot x\otimes x,$$
i.e.,
$$y_K \otimes y_K = (u_Kc_\phi)^2 \cdot \frac{\# \sha(E/K) \cdot {\rm Tam}(E/K)}{\# E(K)_{\rm tors}^2}\cdot x\otimes x.$$
By the well-known Gross-Zagier formula \cite{GZ}, we see that this equality is equivalent to the Birch and Swinnerton-Dyer conjecture. 
\end{remark}


\subsection{The Heegner point main conjecture}




We prove that Conjecture \ref{bd} follows from the Heegner point main conjecture up to $\ZZ_p^\times$. 

\begin{theorem}\label{main}
Assume the Heegner point main conjecture (see Conjecture \ref{IMC} below). Then Conjecture \ref{bd} is true up to $\ZZ_p^\times$, i.e., we have $\cL \in E(K_\infty)\otimes_\ZZ E(K_\infty)\otimes_\ZZ I^\nu$ and there exists $u\in \ZZ_p^\times$ such that 
$$\overline \cL = u\cdot  L_p^2 \cdot \# \sha(E/K)[p^\infty]\cdot {\rm Tam}(E/K) \cdot {\rm Reg}^{\rm BD} \text{ in } \QQ_p\otimes_{\ZZ_p}\left( E(K)\otimes_\ZZ E(K) \otimes_\ZZ Q^\nu\right).$$
\end{theorem}

We will give a proof of Theorem \ref{main} at the end of this section. We shall first state a corollary. 

Let $D_K$ be the discriminant of $K$. For a finite place $v $ of $K$, let $\FF_v$ be the residue field of $v$. 
Combining Theorem \ref{main} with the recent result of Burungale-Castella-Kim \cite[Theorem A]{BCK} on the Heegner point main conjecture, we obtain the following. 

\begin{corollary}\label{BCK}
If we assume
\begin{itemize}
\item $p>3$,
\item $D_K$ is odd and $D_K \neq -3$,
\item the representation $\rho: G_\QQ \to {\rm Aut}(E[p])$ is surjective,
\item $\rho$ is ramified at every $\ell | N$ (in particular, $p\nmid {\rm Tam}(E/K)$),
\item either $N$ is square-free or there are at least two primes $\ell || N$,
\item $p$ is non-anomalous, i.e., $p$ does not divide $\# E(\FF_v)$ for each $v \mid p$,
\end{itemize}
then Conjecture \ref{bd} is true up to $\ZZ_p^\times$. 
\end{corollary}

We recall the formulation of the Heegner point main conjecture. 
For $X\in \{\TT, \bA\}$, we define a Selmer group by
$${\rm Sel}(X):=\ker \left( H^1(\cO_{K,S},X) \to \bigoplus_{v\mid p} H^1(K_v, X^-) \oplus \bigoplus_{v\mid N} H^1_{/{\rm ur}}(K_v,X)\right).$$
Note that we have ${\rm Sel}(\TT)=\widetilde H^1_f(K,\TT)$ by definition. 
We regard the regularized Heegner point $z_\infty$ as an element of ${\rm Sel}(\TT) $ via the Kummer map. 
Let $\iota: \Lambda \to \Lambda ; \ a\mapsto a^\iota$ denote the involution. 

\begin{conjecture}[The Heegner point main conjecture]\label{IMC}
We have
$${\rm char}_\Lambda({\rm Sel}(\TT)/\Lambda\cdot z_\infty) \cdot {\rm char}_\Lambda({\rm Sel}(\TT)/\Lambda\cdot z_\infty)^\iota = {\rm char}_\Lambda(({\rm Sel}(\bA)^\vee)_{\rm tors}).$$
\end{conjecture}

\begin{remark}\label{rem HPMC}
The Heegner point main conjecture is formulated by using the Selmer complex $\widetilde \rgamma_f(K,\TT)$ as follows. 
Let $Q(\Lambda)$ be the quotient field of $\Lambda$. For a $\Lambda$-module $M$, we write $M^\iota$ for the module $M$ on which $\Lambda$ acts via the involution $\iota$. Since we have a canonical duality isomorphism
$$Q(\Lambda)\otimes_\Lambda \widetilde H^2_f(K,\TT)^\ast \simeq Q(\Lambda ) \otimes_\Lambda \widetilde H^1_f(K,\TT)^\iota $$
(see \cite[Lemma 5.9(v)]{ks}), we can regard 
$$z_\infty\otimes z_\infty \in Q(\Lambda) \otimes_\Lambda \left( \widetilde H^1_f(K,\TT) \otimes_\Lambda \widetilde H^2_f(K,\TT)^\ast\right).$$
Let 
$$\pi : Q(\Lambda) \otimes_\Lambda {\det}_\Lambda^{-1}(\widetilde \rgamma_f(K,\TT)) \simeq Q(\Lambda) \otimes_\Lambda \left( \widetilde H^1_f(K,\TT) \otimes_\Lambda \widetilde H^2_f(K,\TT)^\ast\right)$$
be the canonical isomorphism. 
Then the Heegner point main conjecture is equivalent to the following: there exists a $\Lambda$-basis
$$\fz_\infty \in {\det}_\Lambda^{-1}(\widetilde \rgamma_f(K,\TT))$$
such that $\pi(\fz_\infty)=z_\infty\otimes z_\infty$. 
(See \cite[Proposition 5.13]{ks}.) 
\end{remark}

\begin{proof}[Proof of Theorem \ref{main}]
We apply the descent formalism established in Theorem \ref{thm descent}. 
Recall that for a $\Lambda$-module $M$ we denote the localization of $M$ at $I$ by $M_I$. In the present setting, the commutative diagram in Theorem \ref{thm descent} is the following:
$$
\xymatrix{
\displaystyle {\det}_\Lambda^{-1}(\widetilde \rgamma_f(K,\TT)) \ar[r]^-{\pi } \ar@{->>}[dd] &\displaystyle I^{\varrho}\cdot \left( \widetilde H^1_f(K,\TT) \otimes_\Lambda \widetilde H^2_f(K,\TT)^\ast\right)_I\ar[rd]^-{\cD} & \\
 & & \displaystyle  \QQ_p\otimes_{\ZZ_p}\left( E(K) \otimes_\ZZ E(K) \otimes_\ZZ Q^\varrho \right) \\
\displaystyle{\det}_{\ZZ_p}^{-1}(\widetilde \rgamma_f(K,T)) \ar[r]_-{\pi_0  } &\displaystyle \QQ_p\otimes_{\ZZ}\left( {\bigwedge}_\ZZ^{e}E(K) \otimes_\ZZ {\bigwedge}_\ZZ^e E(K) \right) \ar[ru]_-{ {\widetilde R}^{(k_0)}}.&
}
$$
Now assume the Heegner point main conjecture (Conjecture \ref{IMC}) and let $\fz_\infty \in {\det}_\Lambda^{-1}(\widetilde \rgamma_f(K,\TT))$ be a $\Lambda$-basis such that $\pi(\fz_\infty)=z_\infty \otimes z_\infty$ (see Remark \ref{rem HPMC}). 
Note that $\pi$ induces an isomorphism
$$\QQ_p\otimes_{\ZZ_p} {\det}_\Lambda^{-1}(\widetilde \rgamma_f(K,\TT)) \simeq {\rm char}_\Lambda(\widetilde H^2_f(K,\TT)_{\rm tors}) \cdot \left(\QQ_p\otimes_{\ZZ_p}\left(  \widetilde H^1_f(K,\TT) \otimes_\Lambda \widetilde H^1_f(K,\TT)^\iota\right) \right),$$
where ${\rm char}_\Lambda$ denotes the characteristic ideal (see \cite[Lemma 5.9(v)]{ks}). Since ${\rm char}_\Lambda(\widetilde H^2_f(K,\TT)_{\rm tors})$ is divisible by $I^\varrho$ (see Proposition \ref{str h2}), we see that
\begin{equation}\label{z order}
z_\infty \otimes z_\infty \in I^\varrho\cdot \left(\QQ_p\otimes_{\ZZ_p}\left(  \widetilde H^1_f(K,\TT) \otimes_\Lambda \widetilde H^1_f(K,\TT)^\iota\right) \right).
\end{equation}
We set $T_n:={\rm Ind}_{K_n/K}(T)$ and $T_\infty : = \varinjlim_n T_n$. Consider the map
$$\Psi_n : \widetilde H^1_f(K,T_n) \to \widetilde H^1_f(K,T_\infty)\otimes_{\ZZ_p} \ZZ_p[\Gamma_n]; \ x \mapsto \sum_{\sigma \in \Gamma_n}\sigma x \otimes \sigma^{-1}.$$
One sees that the induced map
$$\Psi_\infty:\QQ_p\otimes_{\ZZ_p}\left( \widetilde H^1_f(K,\TT) \otimes_\Lambda \widetilde H^1_f(K,\TT)^\iota \right) \to \QQ_p\otimes_{\ZZ_p} \left( \widetilde H^1_f(K, T_\infty)  \otimes_{\ZZ_p} \widetilde H^1_f(K,T_\infty)\otimes_{\ZZ_p}\Lambda\right)$$
satisfies $\Psi_\infty(z_\infty\otimes z_\infty) = \cL$, where we regard $\QQ_p\otimes_\ZZ E(K_\infty) \subset \QQ_p\otimes_{\ZZ_p} \widetilde H^1_f(K,T_\infty)$ via the Kummer map. (Compare the argument in \cite[\S 5]{AC}.) Since $\Psi_\infty$ is a homomorphism of $\Lambda$-modules, we see by (\ref{z order}) that 
$$ \cL \in  E(K_\infty)\otimes_\ZZ E(K_\infty)\otimes_\ZZ I ^\varrho.$$
Since $\varrho \geq \nu$ by Remark \ref{rem order}, this implies Conjecture \ref{bd}(i). 

If $\varrho > \nu$, then we have $\overline \cL=0$ and Conjecture \ref{bd}(ii) is trivially true. So we may assume $\varrho=\nu$. Then the  argument above shows that
\begin{equation}\label{dl}
\cD(z_\infty\otimes z_\infty) = \overline \cL. 
\end{equation}

We shall deduce Conjecture \ref{bd}(ii) up to $\ZZ_p^\times$. 
Let $\fz_0 \in {\det}_{\ZZ_p}^{-1}(\widetilde \rgamma_f(K,T))$ be the image of the $\Lambda$-basis $\fz_\infty \in {\det}_\Lambda^{-1}(\widetilde \rgamma_f(K,\TT))$. 
Then, by (\ref{dl}) and the commutative diagram above, we have
\begin{equation}\label{comm equality}
\overline \cL = \widetilde R^{(k_0)}(\pi_0(\fz_0)).
\end{equation}
Since $\fz_0$ is a $\ZZ_p$-basis, it is sufficient to prove
\begin{equation}\label{det sha}
\im \pi_0 = L_p^2 \cdot \# \sha(E/K)[p^\infty]\cdot {\rm Tam}(E/K) \cdot   \ZZ_p\otimes_\ZZ \left( {\bigwedge}_{\ZZ}^{e}E(K) \otimes_\ZZ {\bigwedge}_\ZZ^e E(K) \right). 
\end{equation}
We note that $L_p \cdot \ZZ_p = (\prod_{v\mid p}\# E(\FF_v) )\cdot  \ZZ_p $ by the definition of $L_p$ (see (\ref{lpdef})).

For $X \in \{T,A\}$, let $S_X^{\rm str}=S_X^{\rm str}(K)$ be the strict Selmer group as in \cite[(9.6.1)]{nekovar}: 
$$S_X^{\rm str} := \ker \left( H^1(\cO_{K,S},X) \to \bigoplus_{v\mid p}H^1(K_v, X^-) \oplus \bigoplus_{v\mid N} H^1_{/{\rm ur}}(K_v,X)\right).$$
By the exact triangle (\ref{triangle}), we have an exact sequence
\begin{equation*}
 0\to \bigoplus_{v\mid p}H^0(K_v, X^-) \to \widetilde H^1_f(K,X) \to S_X^{\rm str}\to 0. 
 \end{equation*}
(Note that $H^0(K, A) = E(K)[p^\infty] = 0$ by Hypothesis \ref{heegner hyp}(iii).) In particular, we have 
$$\widetilde H^1_f(K,T) = S_T^{\rm str}.$$
Recall that $(-)^\vee:= \Hom_{\ZZ_p}(-,\QQ_p/\ZZ_p)$ denotes the Pontryagin dual. We identify $T^\vee(1)$ with $A$ via the Weil pairing. Since we have
$$\widetilde H^2_f(K,T) \simeq \widetilde H^1_f(K,A)^\vee$$
by duality (see \cite[Proposition 9.7.2(i)]{nekovar}), we have an exact sequence
\begin{equation*}
0\to (S_A^{\rm str})^\vee \to \widetilde H^2_f(K,T) \to \bigoplus_{v \mid p}E(\FF_v)[p^\infty]^\vee \to 0.
\end{equation*}
Hence we obtain a canonical isomorphism
\begin{equation}\label{difference}
{\det}_{\ZZ_p}^{-1}(\widetilde \rgamma_f(K,T)) \simeq L_p\cdot {\det}_{\ZZ_p}(S_T^{\rm str})\otimes_{\ZZ_p} {\det}_{\ZZ_p}^{-1}((S_A^{\rm str})^\vee).
\end{equation}

For $v\in S$, let $H^1_f(K_v, X) \subset H^1(K_v,X)$ denote the Bloch-Kato local condition (see \cite[\S 3]{BK} or \cite[\S 1.3]{R}). Recall that the Bloch-Kato Selmer group is defined by 
$$H^1_f(K,X):=\ker \left( H^1(\cO_{K,S},X) \to \bigoplus_{v\in S} \frac{H^1(K_v,X)}{H^1_f(K_v,X)}\right).$$
By the Poitou-Tate duality (see \cite[Theorem 2.3.4]{MRkoly}), we have an exact sequence
\begin{equation}\label{mr exact}
0\to S_T^{\rm str}\to H^1_f(K,T) \to \bigoplus_{v \mid p} \frac{H^1_f(K_v,T)}{H^1(K_v, T^+)} \oplus \bigoplus_{v\mid N} \frac{H^1_f(K_v,T)}{H^1_{\rm ur}(K_v,T)} \to (S_A^{\rm str})^\vee \to H^1_f(K,A)^\vee \to 0.
\end{equation}
It is well-known that
\begin{equation}\label{tam}
{\rm Tam}(E/K) \cdot \ZZ_p = \left(\prod_{v\mid N} \# \left( \frac{H^1_f(K_v,T)}{H^1_{\rm ur}(K_v,T)}\right) \right) \cdot \ZZ_p.
\end{equation}
Also, we have a natural identification
\begin{equation}\label{h1T}
H^1_f(K,T) = \ZZ_p\otimes_\ZZ E(K)
\end{equation}
and a canonical exact sequence
\begin{equation}\label{h1A}
0 \to \sha(E/K)[p^\infty]^\vee \to H^1_f(K,A)^\vee \to \ZZ_p\otimes_\ZZ E(K)^\ast\to 0.
\end{equation}
By (\ref{difference}), (\ref{mr exact}), (\ref{tam}), (\ref{h1T}) and (\ref{h1A}), we obtain a canonical isomorphism
\begin{multline*}
{\det}_{\ZZ_p}^{-1} (\widetilde \rgamma_f(K,T)) \\
\simeq \left( \prod_{v\mid p} \# \left( \frac{H^1_f(K_v,T)}{H^1(K_v,T^+)}\right)\right) \cdot L_p \cdot \# \sha(E/K)[p^\infty]\cdot {\rm Tam}(E/K)\cdot \ZZ_p\otimes_\ZZ \left( {\bigwedge}_{\ZZ}^{e}E(K) \otimes_\ZZ {\bigwedge}_\ZZ^e E(K) \right).
\end{multline*}

To prove (\ref{det sha}), it is now sufficient to prove
$$ \# \left( \frac{H^1_f(K_v,T)}{H^1(K_v,T^+)}\right) = \# E(\FF_v)[p^\infty]. $$
By local duality, we have
$$\left( \frac{H^1_f(K_v,T)}{H^1(K_v,T^+)}\right)^\vee \simeq  \coker \left( H^1_f(K_v,A) \to \im\left(H^1(K_v,A^+)\to H^1(K_v,A)\right)\right).$$
By the result of Greenberg \cite[Proposition 2.5]{greenberg}, we know that this is isomorphic to $E(\FF_v)[p^\infty]$. Thus we have completed the proof of Theorem \ref{main}.
\end{proof}


\section{A conjecture of Agboola and Castella}\label{second application}

In this section, by a similar argument as in the previous section, we prove that a $p$-adic Birch and Swinnerton-Dyer conjecture for the Bertolini-Darmon-Prasanna $p$-adic $L$-function formulated by Agboola and Castella in \cite[Conjecture 4.2]{AC} follows from the Iwasawa-Greenberg main conjecture up to a $p$-adic unit. Such a result is proved in \cite[Theorem 6.2]{AC} under some hypotheses, but we work under milder assumptions. In particular, we do not assume that $p$ does not divide the Tamagawa factors or that $p$ is non-anomalous.

Let $p$ be an odd prime number. Let $K, E,T,V, A, K_\infty, \Gamma, \Lambda, \TT,\bA$ be as in \S \ref{sec bd}. In this section, {\it we assume Hypothesis \ref{heegner hyp} and that $p$ splits in $K$}. 

\subsection{The BDP $p$-adic $L$-function}

We review the Bertolini-Darmon-Prasanna (``BDP" for short) $p$-adic $L$-function and the Iwasawa-Greenberg main conjecture. 

Let $(p)=\fp \overline \fp$ be the decomposition in $K$. Let $\widehat \ZZ_p^{\rm ur}$ be the completion of the ring of integers of the maximal unramified extension of $\QQ_p$. In \cite{BDP}, Bertolini, Darmon and Prasanna constructed a $p$-adic $L$-function
$$\sL_\fp^{\rm BDP} \in \Lambda^{\rm ur}:=\widehat \ZZ_p^{\rm ur}[[\Gamma]] .$$
(See also \cite{brakocevic}.) It has the following property: if we set
$$L_\fp^{\rm BDP}:=(\sL_\fp^{\rm BDP})^2,$$
then the augmentation map $\Lambda^{\rm ur}\to \widehat \ZZ_p^{\rm ur}$ sends $L_\fp^{\rm BDP}$ to
\begin{equation}\label{BDP formula}
\frac{1}{(u_K c_\phi)^2} \cdot \left( \frac{1-a_p +p}{p}\right)^2 \cdot \log(y_K)^2,
\end{equation}
where $u_K$, $c_\phi$ and $y_K$ are as in \S \ref{formulation bd} and 
$$\log = \log_\fp: \QQ_p\otimes_\ZZ E(K)\to \QQ_p\otimes_{\ZZ_p}  E(K_\fp)^\wedge \to \QQ_p$$
is the formal logarithm (associated to a fixed N\'eron differencial $\omega \in \Gamma(E,\Omega_{E/\QQ}^1)$). 

Next, we review the Iwasawa-Greenberg main conjecture. Let $S$ be the set of places of $K$ which divide $pN\infty$. 
We set
$$H^1_{\overline \fp}(K, \bA):= \ker \left( H^1(\cO_{K,S}, \bA)\to H^1(K_{\overline \fp}, \bA) \oplus \bigoplus_{v\mid N} H^1_{/{\rm ur}}(K_v,\bA) \right).$$
The Iwasawa-Greenberg main conjecture is formulated as follows. 

\begin{conjecture}[The Iwasawa-Greenberg main conjecture]\label{IGMC}
$H^1_{\overline \fp}(K,\bA)^\vee$ is $\Lambda$-torsion and we have
$$\Lambda^{\rm ur}\cdot {\rm char}_\Lambda(H^1_{\overline \fp}(K,\bA)^\vee) = \Lambda^{\rm ur}\cdot L_\fp^{\rm BDP}.$$
\end{conjecture}

\begin{remark}
Conjecture \ref{IGMC} has recently been proved by Burungale-Castella-Kim in \cite[Theorem B]{BCK} under mild hypotheses. In fact, they have proved that Conjecture \ref{IGMC} and the Heegner point main conjecture (Conjecture \ref{IMC}) are equivalent under the assumption $E(K)[p]=0$ (see \cite[Theorem 5.2]{BCK}). 
\end{remark}

\begin{remark}\label{remark torsion}
We know that $H^1_{\overline \fp}(K,\bA)^\vee$ is $\Lambda$-torsion under the assumption $E(K)[p]=0$. This is a consequence of \cite[Lemma 2.3]{nekovar parity} and \cite[Theorem 5.2]{BCK} (see also \cite[Theorem 5.5.2]{CGS}). 
\end{remark}

\subsection{Selmer complexes}

We give another formulation of the Iwasawa-Greenberg main conjecture by using a Selmer complex. 

For $X\in \{T, V,A, \TT, \bA\}$ and $\fq \in \{\fp , \overline \fp\}$, we define $\widetilde \rgamma_{\fq}(K, X)$ by the exact triangle
\begin{equation}\label{modified triangle}
\widetilde \rgamma_{\fq}(K,X)\to \rgamma(\cO_{K,S},X)\to \rgamma(K_\fq,X)\oplus \bigoplus_{v\mid N} \rgamma_{/{\rm ur}}(K_v,X).
\end{equation}
As usual, we use the notation $\widetilde H^i_\fq(K,X):=H^i(\widetilde \rgamma_\fq(K,X))$. Let $\iota: \Lambda\to \Lambda$ be the involution and for a $\Lambda$-module $M$ denote by $M^\iota$ the module on which $\Lambda$ acts via $\iota$. 

\begin{proposition}\label{prop modified selmer}\ 
\begin{itemize}
\item[(i)] There is a canonical isomorphism
$$\widetilde H^i_\fp(K,\TT) \simeq (\widetilde H^{3-i}_{\overline \fp}(K,\bA)^\iota)^\vee.$$
\item[(ii)] $\widetilde \rgamma_\fp(K,\TT)$ is acyclic outside degrees one and two, and $\widetilde H^1_\fp(K,\TT)$ is $\Lambda$-torsion-free. 
\item[(iii)] There is a (non-canonical) isomorphism
$$Q(\Lambda)\otimes_\Lambda \widetilde H^1_\fp(K,\TT)\simeq Q(\Lambda) \otimes_\Lambda \widetilde H^2_\fp(K,\TT),$$
where $Q(\Lambda)$ denotes the quotient field of $\Lambda$. 
\item[(iv)] $\widetilde H^2_\fp(K,\TT)$ is pseudo-isomorphic to $(H^1_{\overline \fp}(K,\bA)^\iota)^\vee$. 
\end{itemize}
\end{proposition}

\begin{proof}
(i) follows from duality (see \cite[(8.9.6.2)]{nekovar}). (We identify $T^\vee(1)$ with $A$ via the Weil pairing.) To prove (ii), we note that $\widetilde H^0_\fp(K,\TT)=0$ follows from $H^0(K,\TT)=0$. The vanishing of $\widetilde H^3_\fp(K,\TT)$ follows from (i) and $H^0(K,\bA)= E(K_\infty)[p^\infty]=0$ (which follows from our assumption $E(K)[p]=0$). Since $H^1(\cO_{K,S},\TT)$ is $\Lambda$-torsion-free (by $E(K)[p]=0$), so is $\widetilde H_\fp^1(K,\TT)$. Hence we have proved (ii). (iii) follows by noting that the Euler-Poincar\'e characteristic of $\widetilde \rgamma_\fp(K,\TT)$ is zero (see \cite[Theorem 7.8.6]{nekovar}). Finally, to prove (iv), it is sufficient to show that $H^1_{\overline \fp}(K,\bA)$ is pseudo-isomorphic to $\widetilde H^1_{\overline \fp}(K,\bA)$. By the exact triangle (\ref{modified triangle}), we have an exact sequence
$$0\to H^0(K_{\overline \fp}, \bA) \to \widetilde H^1_{\overline \fp}(K,\bA) \to H^1_{\overline \fp}(K, \bA) \to 0.$$
So it is sufficient to show that $H^0(K_{\overline \fp},\bA)$ is finite. But this is proved in \cite[Lemma 2.7]{KO}. 
\end{proof}

Since $H^1_{\overline \fp}(K,\bA)^\vee$ is $\Lambda$-torsion (see Remark \ref{remark torsion}), Proposition \ref{prop modified selmer} implies that $Q(\Lambda) \lotimes_\Lambda\widetilde \rgamma_\fp(K,\TT)$ is acyclic and we have a canonical isomorphism
$$\pi: Q(\Lambda)\otimes_\Lambda {\det}_\Lambda^{-1}(\widetilde \rgamma_\fp(K,\TT)) \simeq Q(\Lambda).$$
Also, we have
$$\pi \left( {\det}_\Lambda^{-1}(\widetilde \rgamma_\fp(K,\TT))\right) = {\rm char}_\Lambda(\widetilde H^2_\fp(K,\TT)) = {\rm char}_\Lambda((H^1_{\overline \fp}(K,\bA)^\iota)^\vee).$$
Hence we obtain the following result. 

\begin{proposition}\label{BDP equivalence}
The Iwasawa-Greenberg main conjecture (Conjecture \ref{IGMC}) holds if and only if there is a $\Lambda^{\rm ur}$-basis
$$\fz_\fp \in \Lambda^{\rm ur}\otimes_{\Lambda} {\det}_\Lambda^{-1}(\widetilde \rgamma_\fp(K,\TT))$$
such that
$$\pi(\fz_\fp) = \iota(L_\fp^{\rm BDP}).$$
(Here $\iota: \Lambda^{\rm ur}\to \Lambda^{\rm ur}$ denotes the involution.)
\end{proposition}

\subsection{Selmer groups}\label{def sel}

For later use, we introduce some Selmer groups. 

Let $X \in \{T,V, A\}$. For a place $v$ of $K$, let $H^1_f(K_v,X) \subset H^1(K_v,X)$ be the Bloch-Kato local condition (see \cite[\S 3]{BK} or \cite[\S 1.3]{R}). 
We denote $H^1/H^1_\ast$ by $H^1_{/\ast}$. Also, we use the notation $H^1_\ast (K_p, X):=\bigoplus_{v\mid p}H^1_\ast(K_v,X)$. 
We use the following Selmer groups.
\begin{itemize}
\item The Bloch-Kato Selmer group:
$$H^1_f(K,X) := \ker \left(H^1(\cO_{K,S},X) \to \bigoplus_{v\in S}H^1_{/f}(K_v,X)\right).$$
\item The unramified Bloch-Kato Selmer group:
$$H^1_{f, {\rm ur}}(K,X) := \ker \left(H^1(\cO_{K,S},X)\to H^1_{/f}(K_p,X)\oplus \bigoplus_{v\mid N} H^1_{/{\rm ur}}(K_v,X)\right).$$
\item The strict Selmer group:
$$H^1_0(K,X) := \ker \left(H^1(\cO_{K,S},X)\to H^1(K_p,X)\oplus \bigoplus_{v\mid N} H^1_{/{\rm ur}}(K_v,X)\right).$$
\item For $\fq \in \{ \fp, \overline \fp\}$, the $\fq$-strict Selmer group:
$$H^1_\fq(K,X) := \ker \left(H^1(\cO_{K,S},X)\to H^1(K_\fq,X)\oplus \bigoplus_{v\mid N} H^1_{/{\rm ur}}(K_v,X)\right).$$
\item The relaxed Selmer group:
$$H^1_\emptyset(K,X) := \ker \left(H^1(\cO_{K,S},X)\to  \bigoplus_{v\mid N} H^1_{/{\rm ur}}(K_v,X)\right).$$
\end{itemize}

Note that, by the exact triangle (\ref{modified triangle}), we have identifications
$$\widetilde H^1_\fp(K,T) = H^1_\fp(K,T) \text{ and }\widetilde H^1_\fp(K,V) = H^1_\fp(K,V)$$

\begin{lemma}\label{AC lemma}
For any $\fq \in \{\fp,\overline \fp\}$ we have
$$H^1_\fq(K,V) = \ker (H^1_f(K,V)\to H^1_f(K_\fq,V) )$$
and 
$$\dim_{\QQ_p}(H^1_\fq(K,V)) = {\rm rk}_\ZZ(E(K))-1.$$
\end{lemma}

\begin{proof}
Let $E^K$ be the quadratic twist of $E$ by $K$ and set $r^+ := {\rm rk}_\ZZ(E(\QQ))$ and $r^-:={\rm rk}_\ZZ(E^K(\QQ))$. By Hypothesis \ref{heegner hyp} and the validity of the parity conjecture implies that ${\rm rk}_\ZZ (E(K)) = r^++r^-$ is odd. In particular, we have either $r^+>0$ or $r^->0$. 

Suppose first that $r^+>0$ and $r^->0$. Then we have $H^1(\cO_{K,S},V) = H^1_f(K,V)$ (see \cite[Lemma 5.24(i)]{ks}). So in this case the claim follows by noting that the localization map
$$ H^1_f(K,V)\to H^1_f(K_\fq, V)  $$
is surjective. 

Suppose next that either $r^+ = 0$ or $r^- = 0$. Let
$$\lambda: H^1_f(K,V) \to H^1_f(K_p,V)$$
be the localization map. We claim that $\dim_{\QQ_p}\left( \im \lambda \right) = 1$. 
In fact, 
if $r^-=0$, then we have $H^1_f(K,V) = H^1_f(\QQ,V)$ and $H^1_f(K_p,V) = H^1_f(\QQ_p,V)$. In this case, the map $H^1_f(\QQ,V) \to H^1_f(\QQ_p,V)$ is surjective since $\dim_{\QQ_p}(H^1_f(\QQ,V))=r^+ >0$. This shows the claim when $r^-=0$. The case when $r^+=0$ is treated similarly. 

Since we have proved $\dim_{\QQ_p}(\im \lambda)=1$, we can apply  \cite[Lemma 2.3.2]{skinner} to conclude that
$$H^1_\fq(K,V) = H^1_0(K,V). $$
Applying the snake lemma to the diagram
$$
\xymatrix{
0 \ar[r]& 0 \ar[r] \ar[d] & H^1_f(K,V) \ar@{=}[r] \ar[d]^\lambda & H^1_f(K,V) \ar[r] \ar@{->>}[d] & 0 \\
0 \ar[r] & H^1_f(K_{\overline \fq}, V)\ar[r]& H^1_f(K_p,V) \ar[r] & H^1_f(K_\fq, V) \ar[r] & 0,
}
$$
we obtain an exact sequence
$$0\to H^1_0(K,V)\to \ker (H^1_f(K,V)\to H^1_f(K_\fq,V) ) \to H^1_f(K_{\overline \fq},V)\to \coker \lambda \to 0.$$
Since we have $\dim_{\QQ_p}(H^1_f(K_{\overline \fp},V))=1$ and $\dim_{\QQ_p}(\coker \lambda)=1$, the map $H^1_f(K_{\overline \fq},V)\to \coker \lambda$ must be an isomorphism. Hence we have
$$H^1_\fq(K,V)=H^1_0(K,V) = \ker(H^1_f(K,V)\to H^1_f(K_\fq,V) ).$$
This completes the proof of the lemma. 
\end{proof}


\subsection{The $p$-adic Birch and Swinnerton-Dyer conjecture}

We review the $p$-adic Birch and Swinnerton-Dyer conjecture for the BDP $p$-adic $L$-function formulated by Agboola and Castella in \cite[Conjecture 4.2]{AC}. 


First, we shall define a $p$-adic regulator. Let $I:=\ker(\Lambda \twoheadrightarrow \ZZ_p)$ be the augmentation ideal. 
By Proposition \ref{prop modified selmer}, we see that $\widetilde \rgamma_\fp(K,\TT) \lotimes_\Lambda \Lambda_I$ is represented by a complex of the form (\ref{standard complex}) (with $P=P'$). Also, we know
$$\widetilde \rgamma_\fp(K,\TT)\lotimes_\Lambda \ZZ_p \simeq \widetilde \rgamma_\fp(K,T)$$
by \cite[Proposition 8.10.1]{nekovar}. We set $Q^k:=I^k/I^{k+1}$. Let $\beta_\fp^{(k)}:=\beta^{(k)}(\widetilde \rgamma_\fp(K,\TT))$ be the derived Bockstein map in Definition \ref{main def}:
$$\beta_\fp^{(1)}: \widetilde H^1_\fp(K,T) \to \widetilde H^2_\fp(K,T)\otimes_{\ZZ_p} Q^1,$$
$$\beta_{\fp}^{(k)}: \ker \beta_\fp^{(k-1)} \to \coker \beta_\fp^{(k-1)} \otimes_{\ZZ_p}Q^1.$$

\begin{remark}
Note that we have $\widetilde H^2_\fp(K,V)\simeq \widetilde H^1_{\overline \fp}(K,V)^\ast = H^1_{\overline \fp}(K,V)^\ast$ by duality and the map $\beta_\fp^{(1)}$ induces a pairing
$$H^1_\fp(K,V) \times H^1_{\overline \fp}(K,V)\to \QQ_p\otimes_{\ZZ_p} Q^1.$$
One checks that this pairing coincides with the restriction of the height pairing
$$\widetilde h^{(1)}: H^1_f(K,V)\times H^1_f(K,V)\to \QQ_p\otimes_{\ZZ_p}Q^1$$
defined in \S \ref{review height}. (Note that $H^1_\fp(K,V)$ and $H^1_{\overline \fp}(K,V)$ are subspaces of $H^1_f(K,V)$ by Lemma \ref{AC lemma}.) Moreover, one checks that
\begin{equation}\label{minus one}
{\rm rk}_{\ZZ_p}(\ker \beta_\fp^{(k)}) = {\rm rk}_{\ZZ_p}(\ker \widetilde \beta^{(k)})-1,
\end{equation}
where $\widetilde \beta^{(k)}$ is the derived Bockstein map for $\widetilde \rgamma_f(K,\TT)$. 
\end{remark}

Set
$$e:={\rm rk}_{\ZZ_p}(\widetilde H^2_\fp(K,T)) = {\rm rk}_{\ZZ_p}(\widetilde H^1_\fp (K,T)) \text{ and }e_k:={\rm rk}_{\ZZ_p}(\im \beta_\fp^{(k)}).$$
Since $H^1_{\overline \fp}(K,\bA)^\vee$ is $\Lambda$-torsion, Proposition \ref{str h2} implies that $\tau_{k_0} = \sigma_{k_0} =0$. 
Let
$$R_\fp^{(k_0)}:= R^{(k_0)}(\widetilde \rgamma_\fp(K,\TT)): {\bigwedge}_{\QQ_p}^e \widetilde H^1_\fp(K,V) \otimes_{\QQ_p} {\bigwedge}_{\QQ_p}^e \widetilde H^2_\fp(K,V)^\ast \xrightarrow{\sim} \QQ_p\otimes_{\ZZ_p} Q^\varrho$$
be the derived Bockstein regulator isomorphism in Definition \ref{def der}, where $\varrho  := \sum_{k=1}^{k_0} k e_k$. 
Later we will define a canonical isomorphism
\begin{equation*}\label{mordell compare}
\delta: \QQ_p\otimes_\ZZ \left({\bigwedge}_\ZZ^{s} E(K )\otimes_\ZZ {\bigwedge}_\ZZ^{s} E(K)\right) \simeq {\bigwedge}_{\QQ_p}^e \widetilde H^1_\fp(K,V) \otimes_{\QQ_p} {\bigwedge}_{\QQ_p}^e \widetilde H^2_\fp(K,V)^\ast,
\end{equation*}
where $s:= {\rm rk}_\ZZ (E(K))$ (see Remark \ref{def delta} below). Note that we have $s=e+1$ by Lemma \ref{AC lemma}. 

\begin{remark}\label{rem order 2}
By a similar argument as in Remark \ref{rem order}, one can show that
$$\varrho \geq 2(\max\{r^+,r^-\}-1)=:\nu.$$
In fact, by Lemma \ref{elementary}(iii), we have
$$\varrho  = e+ \sum_{k=1}^{k_0-1}\tau_k \geq e+\tau_1.  $$
By Lemma \ref{AC lemma} and (\ref{minus one}), we have $e= {\rm rk}_\ZZ(E(K))-1$ and $\tau_1 \geq |r^+-r^-|-1$ respectively, and so
$$\varrho \geq {\rm rk}_\ZZ(E(K)) + |r^+-r^-| -2 = \nu.$$
\end{remark}

\begin{remark}
By (\ref{minus one}) and Remark \ref{rem order 2}, one sees that the conjecture of Mazur and Bertolini-Darmon (see Remark \ref{MBD}) implies $\varrho = \nu$. 
\end{remark}


\begin{definition}
Let $\{x_1,\ldots,x_s\}$ be a basis of $E(K)_{\rm tf}$. We define a derived $p$-adic regulator ${\rm Reg}_\fp  \in \QQ_p\otimes_{\ZZ_p} Q^\nu$ by 
$${\rm Reg}_\fp :=   \begin{cases}
R_\fp^{(k_0)} \circ \delta \left((x_1\wedge\cdots \wedge x_s) \otimes (x_1\wedge\cdots \wedge x_s)  \right)   & \text{if $\varrho = \nu$},\\
0 &\text{if $\varrho >\nu$.}
\end{cases}$$
\end{definition}



Let $J:= \ker (\Lambda^{\rm ur}\to \widehat \ZZ_p^{\rm ur})$ be the augmentation ideal of $\Lambda^{\rm ur}$ and set $Q_J^k := J^k/J^{k+1}$. The $p$-adic Birch and Swinnerton-Dyer conjecture is formulated as follows. 

\begin{conjecture}\label{pBSD}
Set $\nu:= 2(\max\{r^+,r^-\}-1)$. 
\begin{itemize}
\item[(i)] We have $L_\fp^{\rm BDP} \in J^\nu$. 
\item[(ii)] Let $\overline L_\fp^{\rm BDP} $ be the image of $L_\fp^{\rm BDP}$ in $Q_J^\nu$. Then we have
$$\overline L_\fp^{\rm BDP} = \left( \frac{1-a_p +p}{p}\right)^2 \cdot \frac{\# \sha(E/K)\cdot {\rm Tam}(E/K)}{\# E(K)_{\rm tors}^2}\cdot {\rm Reg}_\fp.$$
(Here $\sha(E/K)$ is assumed to be finite.)
\end{itemize}
\end{conjecture}

\begin{remark}
When the analytic rank of $E/K$ is one, Conjecture \ref{pBSD} is equivalent to the Birch and Swinnerton-Dyer conjecture (compare  Remark \ref{BD rank one} and \cite[Remark 1.2]{AC}). In this case, we have $s={\rm rk}_\ZZ(E(K)) = 1$ and $\widetilde H^1_\fp(K,V)=\widetilde H^2_\fp(K,V)=0$ (see Lemma \ref{AC lemma}). 
In particular, we have $\varrho=0$. In this case, one checks that the map
$$\delta : \QQ_p\otimes_\ZZ \left(E(K)\otimes_\ZZ E(K)\right) \xrightarrow{\sim} \QQ_p$$
is given by $\delta(a \otimes b) =\log(a)\log(b)$ (see Remark \ref{logab} below). So we have
$${\rm Reg}_\fp = \log(x)^2$$
with a basis $x \in E(K)_{\rm tf}$. By the interpolation property (\ref{BDP formula}) of the BDP $p$-adic $L$-function, we see that Conjecture \ref{pBSD} in this case is equivalent to
$$\log(y_K)^2  = (u_Kc_\phi)^2 \cdot \frac{\# \sha(E/K) \cdot {\rm Tam}(E/K)}{\# E(K)_{\rm tors}^2}\cdot \log(x)^2.$$
By the Gross-Zagier formula, this is equivalent to the Birch and Swinnerton-Dyer conjecture. 
\end{remark}

\subsection{Relation with the main conjecture}

In this subsection, we prove the following theorem. Note that we assume only Hypothesis \ref{heegner hyp} and that $p$ splits in $K$. 

\begin{theorem}\label{improved AC}
The Iwasawa-Greenberg main conjecture (Conjecture \ref{IGMC}) implies Conjecture \ref{pBSD} up to $(\widehat \ZZ_p^{\rm ur})^\times$, i.e., we have $L_\fp^{\rm BDP} \in J^\nu$ and there exists $u \in (\widehat \ZZ_p^{\rm ur})^\times$ such that
$$\overline L_\fp^{\rm BDP} = u\cdot \left(\frac{1-a_p+p}{p}\right)^2\cdot \# \sha(E/K)[p^\infty]\cdot {\rm Tam}(E/K)\cdot {\rm Reg}_\fp.$$
\end{theorem}


\begin{remark}\label{remark AC}
Theorem \ref{improved AC} is proved by Agboola and Castella in \cite[Theorem 6.2]{AC} under additional hypotheses (including $p\nmid {\rm Tam}(E/K)$ and $a_p \not \equiv 1$ (mod $p$)).
\end{remark}

The rest of this section is devoted to the proof of Theorem \ref{improved AC}. 

We use the Selmer groups defined in \S \ref{def sel}. 
By the Poitou-Tate duality (see \cite[Theorem 2.3.4]{MRkoly}), we have the following exact sequences:
\begin{equation}\label{PT1}
0\to H^1_0 (K,T)\to H^1_\fp(K,T)\to H^1(K_{\overline \fp},T) \to H^1_\emptyset(K,A)^\vee \to H^1_{\overline \fp}(K,A)^\vee \to 0,
\end{equation}
\begin{equation}\label{PT2}
0\to H^1_0 (K,T)\to H^1_{f, {\rm ur}}(K,T)\to H^1_{f}(K_p,T) \to H^1_\emptyset(K,A)^\vee \to H^1_{f,{\rm ur}}(K,A)^\vee \to 0,
\end{equation}
\begin{equation}\label{PT3}
0\to H^1_{f, {\rm ur}} (K,T)\to H^1_f(K,T)\to \bigoplus_{v\mid N} \frac{H^1_f(K_v,T)}{H^1_{\rm ur}(K_v,T)} \to H^1_{f,{\rm ur}}(K,A)^\vee \to H^1_f(K,A)^\vee \to 0.
\end{equation}
Using these sequences, we obtain
\begin{eqnarray}\label{PTisom}
&&{\det}_{\ZZ_p}(H^1_\fp(K,T)) \otimes_{\ZZ_p}{\det}_{\ZZ_p}^{-1}(H^1_{\overline \fp}(K,A)^\vee) \nonumber \\
&\stackrel{(\ref{PT1})}{\simeq}& {\det}_{\ZZ_p}(H^1(K_{\overline \fp},T)) \otimes_{\ZZ_p} {\det}_{\ZZ_p}(H^1_0(K,T))\otimes_{\ZZ_p} {\det}_{\ZZ_p}^{-1}(H^1_\emptyset(K,A)^\vee) \nonumber \\
&\stackrel{(\ref{PT2})}{\simeq}&{\det}_{\ZZ_p}(H^1(K_{\overline \fp},T)) \otimes_{\ZZ_p}  {\det}_{\ZZ_p}(H^1_{f,{\rm ur}}(K,T)) \otimes_{\ZZ_p} {\det}_{\ZZ_p}^{-1}(H^1_{f}(K_p,T)) \otimes_{\ZZ_p} {\det}_{\ZZ_p}^{-1}(H^1_{f,{\rm ur}}(K,A)^\vee) \nonumber \\
&\stackrel{(\ref{PT3})}{\simeq}&{\rm Tam}(E/K)\cdot {\det}_{\ZZ_p}(H^1(K_{\overline \fp},T)) \otimes_{\ZZ_p} {\det}_{\ZZ_p}^{-1}(H^1_f(K_p,T)) \otimes_{\ZZ_p}{\det}_{\ZZ_p}(H^1_f(K,T)) \otimes_{\ZZ_p}  {\det}_{\ZZ_p}^{-1}(H^1_f(K,A)^\vee) .\nonumber
\end{eqnarray}
(Here we used the identification (\ref{tam}).)

Since we have $\widetilde H^2_\fp(K,T) \simeq \widetilde H^1_{\overline \fp}(K,A)^\vee$ by duality (see \cite[(8.9.6.2)]{nekovar}), the exact triangle (\ref{modified triangle}) induces an exact sequence
$$0\to H^1_{\overline \fp}(K,A)^\vee \to \widetilde H^2_\fp(K,T) \to H^0(K_{\overline \fp},A)^\vee = E(K_{\overline \fp})[p^\infty]^\vee \to 0.$$
Hence we obtain an isomorphism
\begin{equation*}\label{det fp isom}
{\det}_{\ZZ_p}^{-1}(\widetilde \rgamma_\fp(K,T)) \simeq  {\det}_{\ZZ_p}^{-1}(E(K_{\overline \fp})[p^\infty]^\vee)\otimes_{\ZZ_p}  {\det}_{\ZZ_p}(H^1_\fp(K,T))\otimes_{\ZZ_p} {\det}_{\ZZ_p}^{-1}(H^1_{\overline \fp}(K,A)^\vee).
\end{equation*}
Combining this isomorphism with the isomorphism in the previous paragraph, we obtain
\begin{multline}\label{resulting isom}
{\det}_{\ZZ_p}^{-1}(\widetilde \rgamma_\fp(K,T)) \simeq {\rm Tam}(E/K)\cdot {\det}_{\ZZ_p}^{-1}(E(K_{\overline \fp})[p^\infty]^\vee)\otimes_{\ZZ_p}  {\det}_{\ZZ_p}(H^1(K_{\overline \fp},T)) \otimes_{\ZZ_p} {\det}_{\ZZ_p}^{-1}(H^1_f(K_p,T))\\
 \otimes_{\ZZ_p} {\det}_{\ZZ_p}(H^1_f(K,T))\otimes_{\ZZ_p} {\det}_{\ZZ_p}^{-1}(H^1_f(K,A)^\vee).
\end{multline}

\begin{lemma}\label{calculation lemma}\ 
\begin{itemize}
\item[(i)] There is a canonical isomorphism
$$ {\det}_{\ZZ_p}^{-1}(E(K_{\overline \fp})[p^\infty]^\vee)\otimes_{\ZZ_p} {\det}_{\ZZ_p}(H^1(K_{\overline \fp},T)) \simeq \ZZ_p.$$
\item[(ii)] The formal logarithm induces an isomorphism
$${\det}_{\ZZ_p}^{-1}(H^1_f(K_p,T)) \simeq  \left(\frac{1-a_p+p}{p}\right)^2\cdot \ZZ_p.$$
\end{itemize}
\end{lemma}

\begin{proof}
(i) By the exact sequence
$$ 0\to H^1_f(K_{\overline \fp},T) \to H^1(K_{\overline \fp},T) \to H^1_{/f}(K_{\overline \fp}, T) \to 0$$
and the identifications
$$H^1_f(K_{\overline \fp},T) = E(K_{\overline \fp})^\wedge \text{ and }H^1_{/f}(K_{\overline \fp},T) = E(K_{\overline \fp})^{\wedge,\ast},$$
where $(-)^\wedge$ denotes the $p$-completion, we obtain a canonical isomorphism
$${\det}_{\ZZ_p}(H^1(K_{\overline \fp},T))\simeq {\det}_{\ZZ_p}(E(K_{\overline \fp})^\wedge)\otimes_{\ZZ_p} {\det}_{\ZZ_p}(E(K_{\overline \fp})^{\wedge,\ast}).$$
Combining this isomorphism with
$$ {\det}_{\ZZ_p}(E(K_{\overline \fp})^{\wedge,\ast})\otimes_{\ZZ_p} {\det}_{\ZZ_p}^{-1}(E(K_{\overline \fp})[p^\infty]^\vee) = {\det}_{\ZZ_p}(\rhom_{\ZZ_p}(E(K_{\overline \fp})^\wedge, \ZZ_p)) = {\det}_{\ZZ_p}^{-1}(E(K_{\overline \fp})^\wedge) ,$$
we obtain a canonical isomorphism
$${\det}_{\ZZ_p}^{-1}(E(K_{\overline \fp})[p^\infty]^\vee) \otimes_{\ZZ_p}{\det}_{\ZZ_p}(H^1(K_{\overline \fp},T)) \simeq {\det}_{\ZZ_p}(E(K_{\overline \fp})^\wedge)\otimes_{\ZZ_p} {\det}_{\ZZ_p}^{-1}(E(K_{\overline \fp})^\wedge).$$
Since the right hand side is canonically isomorphic to $\ZZ_p$ by evaluation, the claim follows. 

(ii) For each $v\mid p$, we have a natural identification
$$H^1_f(K_v,T) = E(K_v)^\wedge.$$
Since $K_v=\QQ_p$, it is sufficient to show that the formal logarithm $\log: E_1(\QQ_p) \xrightarrow{\sim} p\ZZ_p$ induces an isomorphism
$${\det}_{\ZZ_p}^{-1}(E(\QQ_p)^\wedge) \xrightarrow{\sim} \left(\frac{1-a_p+p}{p}\right)\cdot \ZZ_p.$$
But this is easily proved by noting that $1-a_p + p = \# E(\FF_p) = [E(\QQ_p):E_1(\QQ_p)]$. 
\end{proof}

By (\ref{resulting isom}) and Lemma \ref{calculation lemma}, we obtain an isomorphism
$${\det}_{\ZZ_p}^{-1}(\widetilde \rgamma_\fp(K,T))\simeq \left(\frac{1-a_p+p}{p}\right)^2\cdot {\rm Tam}(E/K)\cdot {\det}_{\ZZ_p}(H^1_f(K,T))\otimes_{\ZZ_p}{\det}_{\ZZ_p}^{-1}(H^1_f(K,A)^\vee).$$
Finally, by using (\ref{h1T}) and (\ref{h1A}), we obtain 
\begin{equation}\label{final isom}
{\det}_{\ZZ_p}^{-1}(\widetilde \rgamma_\fp(K,T))\simeq \left(\frac{1-a_p+p}{p}\right)^2\cdot \#\sha(E/K)[p^\infty]\cdot {\rm Tam}(E/K)\cdot \ZZ_p \otimes_\ZZ \left( {\bigwedge}_\ZZ^s E(K) \otimes_\ZZ {\bigwedge}_\ZZ^s E(K)\right).
\end{equation}

\begin{remark}\label{def delta}
The (inverse of the) isomorphism (\ref{final isom}) induces an isomorphism
$$\delta: \QQ_p\otimes_\ZZ \left({\bigwedge}_\ZZ^{s} E(K )\otimes_\ZZ {\bigwedge}_\ZZ^{s} E(K)\right) \xrightarrow{\sim} {\det}_{\QQ_p}^{-1}(\widetilde \rgamma_\fp(K,V)) = {\bigwedge}_{\QQ_p}^e \widetilde H^1_\fp(K,V) \otimes_{\QQ_p} {\bigwedge}_{\QQ_p}^e \widetilde H^2_\fp(K,V)^\ast.$$
This isomorphism is directly obtained by using the Poitou-Tate duality sequences
$$0\to H^1_0(K,V)\to H^1_\fp(K,V)\to H^1(K_{\overline \fp},V)\to H^1_\emptyset(K,V)^\ast \to H^1_{\overline \fp}(K,V)^\ast \to 0,$$
$$0\to H^1_0(K,V) \to \QQ_p\otimes_\ZZ E(K) \to H^1_f(K_p,V) \to  H^1_\emptyset(K,V)^\ast \to \QQ_p\otimes_\ZZ E(K)^\ast \to 0$$
and the identifications
$${\bigwedge}_{\QQ_p}^2 H^1(K_{\overline \fp},V) =\QQ_p\text{ and }{\bigwedge}_{\QQ_p}^2 H^1_f(K_p,V) = \QQ_p$$
as in Lemma \ref{calculation lemma}. 
\end{remark}

\begin{remark}\label{logab}
When $s={\rm rk}_\ZZ(E(K))=1$, we have $\widetilde H^1_\fp(K,V) = H^1_\fp(K,V)=0$ and $\widetilde H^2_\fp(K,V)^\ast \simeq \widetilde H^1_{\overline \fp}(K,V) =H^1_{\overline \fp}(K,V)=0$ (see Lemma \ref{AC lemma}). In this case, the map
$$\delta: \QQ_p\otimes_\ZZ \left(E(K)\otimes_\ZZ E(K)\right) \xrightarrow{\sim} \QQ_p$$
is explicitly given by $\delta(a\otimes b) = \log(a)\log(b)$. In fact, the isomorphism
$$\QQ_p\otimes_\ZZ \left(E(K)\otimes_\ZZ E(K)\right)  \simeq {\bigwedge}_{\QQ_p}^2 H^1_f(K_p,V)\otimes_{\QQ_p} {\bigwedge}_{\QQ_p}^2 H^1(K_{\overline \fp},V)^\ast,$$
induced by the sequences above sends $a\otimes b$ to an element of the form
$$(\ell_p(a)\wedge c)\otimes (d \wedge \ell_{\overline \fp}(b)),$$
where $\ell_p: E(K)\to H^1_f(K_p,V)$ and $\ell_{\overline \fp}: E(K)\to H^1(K_{\overline \fp},V) \simeq H^1(K_{\overline \fp},V)^\ast$ are the localization maps and the elements $c$ and $d$ are chosen as follows. We choose $d \in H^1(K_{\overline \fp},V)^\ast \simeq H^1(K_{\overline \fp},V)$ so that its image in $H^1_{/f}(K_{\overline \fp},V)\simeq H^1_f(K_{\overline \fp},V)^\ast$ is dual to $\ell_{\overline \fp}(b) \in H^1_f(K_{\overline \fp},V)$. We choose $c\in H^1_f(K_p,V) = H^1_f(K_\fp,V)\oplus H^1_f(K_{\overline \fp},V)$ so that the $\fp$-component is zero and the $\overline \fp$-component is dual to the image of $d$ in $H^1_f(K_{\overline \fp},V)^\ast$, namely, $\ell_{\overline \fp}(b)$. Then the isomorphism ${\bigwedge}_{\QQ_p}^2 H^1(K_{\overline \fp},V)^\ast \simeq \QQ_p$ as in Lemma \ref{calculation lemma}(i) sends $d\wedge \ell_{\overline \fp}(b)$ to 1. Also, the isomorphism ${\bigwedge}_{\QQ_p}^2 H^1_f(K_p,V)\simeq \QQ_p$ as in Lemma \ref{calculation lemma}(ii) sends $\ell_p(a)\wedge c$ to
$$\det \begin{pmatrix}
\log_\fp(a) & 0 \\
\log_{\overline \fp}(a) & \log_{\overline \fp}(b)
\end{pmatrix} = \log(a)\log(b).$$
(Note that, since ${\rm rk}_\ZZ(E(K))=1$, we have either $\QQ_p\otimes_\ZZ E(\QQ)=\QQ_p\otimes_\ZZ E(K)$ or $\QQ_p\otimes_\ZZ E^K(\QQ)= \QQ_p\otimes_\ZZ E(K)$, where $E^K$ is the quadratic twist of $E$ by $K$, and so the maps $\log_\fp(=\log)$ and $\log_{\overline \fp}$ are the same.) This proves the claim. 
\end{remark}

We now prove Theorem \ref{improved AC}. 

\begin{proof}[Proof of Theorem \ref{improved AC}]
By Theorem \ref{thm descent}, we have a commutative diagram
$$
\xymatrix{
\displaystyle {\det}_\Lambda^{-1}(\widetilde \rgamma_\fp(K,\TT)) \ar[r]^-{\pi } \ar@{->>}[dd] &\displaystyle I^\varrho \cdot \Lambda_I \ar@{->>}[rd] & \\
 & & \displaystyle  \QQ_p\otimes_{\ZZ_p}Q^\varrho \\
\displaystyle{\det}_{\ZZ_p}^{-1}(\widetilde \rgamma_\fp(K,T)) \ar[r]_-{\pi_0  } &\displaystyle {\bigwedge}_{\QQ_p}^e \widetilde H^1_\fp(K,V) \otimes_{\QQ_p} {\bigwedge}_{\QQ_p}^e \widetilde H^2_\fp(K,V)^\ast \ar[ru]_-{ R_\fp^{(k_0)}}.& 
}
$$
Assume the Iwasawa-Greenberg main conjecture (Conjecture \ref{IGMC}). Then, by Proposition \ref{BDP equivalence}, there is a $\Lambda^{\rm ur}$-basis
$$\fz_\fp \in \Lambda^{\rm ur}\otimes_\Lambda {\det}_\Lambda^{-1}(\widetilde \rgamma_\fp(K,\TT))$$
such that $\pi(\fz_\fp)=\iota(L_\fp^{\rm BDP})$. This shows that $L_\fp^{\rm BDP} \in J^\varrho$. Since we have $\varrho \geq \nu$ by Remark \ref{rem order 2}, we have $L_\fp^{\rm BDP} \in J^\nu$. Thus we have proved Conjecture \ref{pBSD}(i). If $\varrho > \nu$, then Conjecture \ref{pBSD}(ii) is trivially true, so we may assume $\varrho=\nu$. Note the image of $\iota(L_\fp^{\rm BDP})$ in $Q_J^\nu$ coincides with $\overline L_\fp^{\rm BDP}$ up to sign. 
Let $\overline \fz_\fp \in \widehat \ZZ_p^{\rm ur}\otimes_{\ZZ_p} {\det}_{\ZZ_p}^{-1}(\widetilde \rgamma_\fp(K,T))$ be the image of $\fz_\fp$, which is a $\widehat \ZZ_p^{\rm ur}$-basis. By the commutative diagram above, we have
$$\overline L_\fp^{\rm BDP} = \pm R_\fp^{(k_0)}(\pi_0(\overline \fz_\fp)).$$
So it is sufficient to show that
$$\widehat \ZZ_p^{\rm ur}\cdot \delta^{-1}(\pi_0(\overline \fz_\fp)) = \left(\frac{1-a_p+p}{p}\right)^2\cdot \# \sha(E/K)[p^\infty]\cdot {\rm Tam}(E/K)\cdot \widehat \ZZ_p^{\rm ur}\otimes_\ZZ \left({\bigwedge}_\ZZ^s E(K)\otimes_\ZZ {\bigwedge}_\ZZ^s E(K)\right),$$
where $\delta$ is as in Remark \ref{def delta}. However, this follows immediately from (\ref{final isom}). Thus we have proved the theorem. 
\end{proof}

\section{Derivatives of Euler systems}\label{ks application}

In this section, we extend conjectures and results in \cite{ks} into the derived setting. 

\subsection{Notation and hypotheses}
Let $p$ be an odd prime number. Let $\cO$ be the ring of integers of a finite extension of $\QQ_p$. Let $K$ be a number field and fix a $\ZZ_p$-extension $K_\infty/K$. Let $K_n$ denote the $n$-th layer of $K_\infty/K$. We set
$$\Gamma_n:=\Gal(K_n/K), \ \Gamma:=\Gal(K_\infty/K) , \ \Lambda_n :=\cO[\Gamma_n]\text{ and }\Lambda:=\cO[[\Gamma]].$$
Let $M$ be a pure motive defined over $K$ with coefficients in $\cO$, and $T$ the $p$-adic \'etale realization of $M$ (so $T$ is a free $\cO$-module of finite rank with a continuous $\cO$-linear action of $G_K$). As in \S \ref{sec bd}, let 
$$\TT:=T \otimes_{\cO}\Lambda$$
be the induced module. Let $S_\infty(K)$ be the set of infinite places of $K$. 
Fix a finite set $S $ of places of $K$ containing $S_\infty(K)$, all places above $p$, and all places at which $T$ ramifies. Fix also a finite set $\Sigma$ of places of $K$ which is disjoint from $S$. (We may take $\Sigma = \emptyset$.) We set $T^\ast(1):=\Hom_{\ZZ_p}(T,\ZZ_p(1))$ and 
$$Y_K(T^\ast(1)):=\bigoplus_{v\in S_\infty(K)}H^0(K_v,T^\ast(1)).$$
We define the {\it basic rank} of $T$ by
\begin{equation}\label{basic}
r=r_T:= {\rm rk}_\cO\left( Y_K(T^\ast(1))\right).
\end{equation}
For a set $U$ of places of $K$ and an extension $L/K$, we write $U_L$ for the set of places of $L$ which lie above a place in $U$. 

Throughout this section, we assume the following hypothesis. 

\begin{hypothesis}[{\cite[Hypothesis 2.3]{ks}}]\label{hyp}\
For every $n$, we have
\begin{itemize}
\item[(i)] $H^0(K_n,T)=0$,
\item[(ii)] either $\Sigma$ is non-empty or the $\cO$-module $H^1(\cO_{K_n,S},T)$ is free,
\item[(iii)] $H^0(K_{n,w},T)=0$ for any $w\in \Sigma_{K_n}$. 
\end{itemize}
\end{hypothesis}

Recall that a ``$\Sigma$-modified $S$-cohomology complex"
$$\rgamma_\Sigma(\cO_{K,S},\TT)$$
is defined in \cite{sbA}. Suppose that a system of special elements (Euler system)
$$c=(c_{K_n})_n \in \varprojlim_n {\bigcap}_{\Lambda_n}^r H_\Sigma^1(\cO_{K_n,S},T) ={\bigcap}_\Lambda^r H_\Sigma^1(\cO_{K,S},\TT)$$
is given (see \cite[Definition 2.2]{ks}). The following examples of $c$ should be kept in mind. 

\begin{example}\label{ex}\
\begin{itemize}
\item[(i)] Let $K_\infty/K$ be the cyclotomic $\ZZ_p$-extension. Let $\chi: G_K\to \overline \QQ^\times$ a character of finite order such that the field $L:=\overline \QQ^{\ker \chi}$ is disjoint from $K_\infty$. We set $\cO:=\ZZ_p[\im \chi]$ and consider the representation $T:=\cO(1)\otimes \chi^{-1}$, which is a free $\cO$-module of rank one on which $G_K$ acts via $\chi_{\rm cyc}\chi^{-1}$. Here $\chi_{\rm cyc}: G_K\to \ZZ_p^\times$ denotes the cyclotomic character. We set $V:=\{v\in S_\infty(K) \mid \text{$v$ splits completely in $L$}\}$. In this case, Hypothesis \ref{hyp} is satisfied with $r=\# V$ and a non-empty set $\Sigma$. 
We set $L_\infty:=L K_\infty$ and $L_n:=L K_n$. 
The $\chi$-component of the conjectural Rubin-Stark element
$$\eta^{\rm RS}=(\eta_{L_n/K,S,\Sigma}^{V,\chi})_n \in \varprojlim_n {\bigcap}_{\Lambda_n}^r H^1_\Sigma(\cO_{K_n,S},T) ={\bigcap}_\Lambda^r H^1_\Sigma(\cO_{K,S},\TT)$$
is a canonical Euler system. 
\item[(ii)] Let $K=\QQ$ and $\cO=\ZZ_p$. Let $E$ be an elliptic curve over $\QQ$ and set $T:=T_p(E)$. We assume that $E(\QQ)[p]=0$. Then Hypothesis \ref{hyp} is satisfied with $r=1$ and $\Sigma=\emptyset$. As a canonical Euler system one can take Kato's Euler system
$$z^{\rm Kt}=(z_{\QQ_n}^{\rm Kt})_n \in \varprojlim_n H^1(\cO_{\QQ_n,S},T)=H^1(\ZZ_S,\TT)={\bigcap}_\Lambda^1 H^1(\ZZ_S,\TT).$$
\item[(iii)] Let $K$ be an imaginary quadratic field and $K_\infty/K$ the anticyclotomic $\ZZ_p$-extension. We take $\cO=\ZZ_p$. Let $E$ be an elliptic curve over $\QQ$ such that $K$ satisfies the Heegner hypothesis for $E$. We set $T:=T_p(E)$ and assume $E(K)[p]=0$. Then Hypothesis \ref{hyp} is satisfied with $r=2$ and $\Sigma=\emptyset$. If we further assume that $E$ has good ordinary reduction at $p$ and $K/\QQ$ is unramified at $p$, then the ``$\Lambda$-adic Heegner element"
$$z^{\rm Hg} \in  Q(\Lambda)\otimes_\Lambda {\bigcap}_\Lambda^2 H^1(\cO_{K,S},\TT)$$
constructed in \cite[Definition 5.14]{ks} is an Euler system, which is canonical up to $\Lambda^\times$. Here $Q(\Lambda)$ denotes the quotient field of $\Lambda$. (Conjecturally, the element $z^{\rm Hg}$ belongs to ${\bigcap}_\Lambda^2 H^1(\cO_{K,S},\TT)$.) 
\end{itemize}
\end{example}

The following standard hypothesis is also assumed throughout this section. 

\begin{hypothesis}[The weak Leopoldt conjecture]\label{wl}
$H^2_\Sigma(\cO_{K,S},\TT)$ is $\Lambda$-torsion. 
\end{hypothesis}

\begin{remark}
In the settings of Example \ref{ex}(i) and (ii), we know that Hypothesis \ref{wl} is satisfied (see \cite[(9.2.2)]{nekovar} and \cite[Theorem 12.4(1)]{katoasterisque}). In the setting of Example \ref{ex}(iii), Hypothesis \ref{wl} is proved in \cite[Theorem 5.4]{bertolinileop} under mild hypotheses. 
\end{remark}

\subsection{The $p$-adic leading term conjecture}

We first formulate a refinement of the conjecture in \cite[Conjecture 4.7]{ks}, which constitutes a generalization of the Gross-Stark conjecture \cite{Gp} and the Mazur-Tate-Teitelbaum conjecture \cite{MTT}. 

Under Hypothesis \ref{hyp}, we know that $\rgamma_\Sigma(\cO_{K,S},\TT)\lotimes_\Lambda \Lambda_I$ is represented by a complex of the form (\ref{standard complex}). So we can apply the construction in \S \ref{sec derived} with $C= \rgamma_\Sigma(\cO_{K,S},\TT)$. 
Let
$$\beta^{(1)}=\beta^{(1)}(\rgamma_\Sigma(\cO_{K,S}, \TT)): H^1_\Sigma(\cO_{K,S},T) \to H^2_\Sigma(\cO_{K,S},T)\otimes_\cO Q^1$$
be the Bockstein map and 
$$\beta^{(k)}=\beta^{(k)}(\rgamma_\Sigma(\cO_{K,S},\TT)): \ker \beta^{(k-1)} \to \coker \beta^{(k-1)}\otimes_\cO Q^1.$$
the $k$-th derived Bockstein map (see Definition \ref{main def}). 

We set 
$$e:={\rm rk}_\cO(H^2_\Sigma(\cO_{K,S},T)).$$
Note that we have
$${\rm rk}_\cO(H^1_\Sigma(\cO_{K,S},T)) = r+e,$$
where $r$ is the basic rank defined in (\ref{basic}). 

Let $k_0$ be as in Definition \ref{def k0}. 
Hypothesis \ref{wl} implies that ${\rm rk}_\cO(\coker \beta^{(k_0)})=0$ (see Proposition \ref{str h2}). So the $k_0$-th derived Bockstein regulator map in (\ref{def der map}) becomes a map
$$R^{(k_0)}: \QQ_p \otimes_{\ZZ_p}\left( {\bigwedge}_\cO^{r+e} H_\Sigma^1(\cO_{K,S},T)\otimes_\cO{\bigwedge}_\cO^e H^2_\Sigma(\cO_{K,S},T)^\ast\right) \to \QQ_p \otimes_{\ZZ_p}  {\bigwedge}_\cO^r H^1_\Sigma(\cO_{K,S},T) \otimes_\cO  Q^{\varrho},$$
where $\varrho:=e+\sum_{k=1}^{k_0-1} {\rm rk}_\cO (\coker \beta^{(k)})$ (see Lemma \ref{elementary}(iii)).

Let
$$\cD: I^{\varrho} \cdot {\bigwedge}_{\Lambda_I}^r H^1_\Sigma(\cO_{K,S},\TT)_I \to  \QQ_p\otimes_{\ZZ_p}{\bigwedge}_\cO^r H^1_\Sigma(\cO_{K,S},T) \otimes_\cO Q^{\varrho}$$
be the map in Theorem \ref{thm descent}.


Finally, let 
$$ \eta= \eta(T)  \in\CC_p \otimes_{\ZZ_p}\left( {\bigwedge}_\cO^{r+e} H_\Sigma^1(\cO_{K,S},T)\otimes_\cO{\bigwedge}_\cO^e H^2_\Sigma(\cO_{K,S},T)^\ast\right)$$
be the extended special element attached to $T$ (see \cite[Definition 2.11]{ks}). 

\begin{conjecture}\label{leading}\
\begin{itemize}
\item[(i)] We have
$$c \in I^{\varrho}\cdot {\bigwedge}_{\Lambda_I}^r H_\Sigma^1(\cO_{K,S},\TT)_I.$$
\item[(ii)] If (i) holds, then we have
$$\cD(c)=R^{(k_0)}( \eta) \text{ in }\CC_p \otimes_{\ZZ_p} {\bigwedge}_\cO^r H^1_\Sigma(\cO_{K,S},T) \otimes_\cO Q^{\varrho}.$$
\end{itemize}
\end{conjecture}

\begin{remark}
Conjecture \ref{leading} is a refinement of \cite[Conjecture 4.7]{ks}. In fact, when $k_0=1$, these conjectures are equivalent. (In this case, note that $\varrho = e$.) When $k_0>1$, the equality in \cite[Conjecture 4.7]{ks} becomes ``$0=0$", while our equality in Conjecture \ref{leading}(ii) is always non-trivial, since by definition $R^{(k_0)}$ is injective and $\eta$ is non-zero. 
\end{remark}

\begin{remark}
In each of the settings of Example \ref{ex}, one can formulate Conjecture \ref{leading} in a more explicit way. In the case of Example \ref{ex}(i) (resp. (ii)), see \cite[Conjecture 4.2]{bks2} (resp. \cite[Proposition 4.15]{bks4}). (In these cases, we conjecturally have $k_0=1$, and our derived regulators should coincide with the usual regulators.) In the case of Example \ref{ex}(iii), we give an explicit interpretation in \S \ref{anti case}. 

\end{remark}

\subsection{The Iwasawa main conjecture}

The following formulation of the Iwasawa main conjecture is given in \cite[Conjecture 3.9]{ks}. 

\begin{conjecture}[The Iwasawa main conjecture]\label{ks main conj}
We have
$${\rm char}_\Lambda\left( {\bigcap}_\Lambda^r H^1_\Sigma(\cO_{K,S},\TT) / \Lambda \cdot c\right) = {\rm char}_\Lambda(H^2_\Sigma(\cO_{K,S},\TT)).$$
\end{conjecture}

The following result is a refinement of \cite[Theorem 4.11]{ks}. 

\begin{theorem}\label{strategy}
Assume the Iwasawa main conjecture (Conjecture \ref{ks main conj}) and Conjecture \ref{leading}. Then the Tamagawa number conjecture for $(M^\ast(1),\cO)$ is true. 
\end{theorem}

\begin{remark}
When $k_0=1$, Theorem \ref{strategy} is the same as \cite[Theorem 4.11]{ks}. In fact, the assumption ``${\rm Boc}_\infty$ is non-zero" in loc. cit. is equivalent to $k_0=1$. 
\end{remark}

\begin{proof}[Proof of Theorem \ref{strategy}]
By Theorem \ref{thm descent}, we have the following commutative diagram:
$$
\scriptsize
\xymatrix@C=12pt@R=30pt{
\displaystyle {\det}_\Lambda^{-1}( \rgamma_\Sigma(\cO_{K,S},\TT)) \ar[r]^-{\pi } \ar@{->>}[dd] &\displaystyle I^{\varrho}\cdot {\bigwedge}_{\Lambda_I}^r H^1_\Sigma(\cO_{K,S},\TT)_I \ar[rd]^-{\cD} & \\
 & & \displaystyle  \CC_p\otimes_{\ZZ_p}\left( {\bigwedge}_\cO^r H^1_\Sigma(\cO_{K,S},T)\otimes_\cO Q^\varrho \right) \\
\displaystyle {\det}_{\cO}^{-1}( \rgamma_\Sigma(\cO_{K,S},T)) \ar[r]_-{\pi_0 } &\displaystyle \CC_p\otimes_{\ZZ_p}\left( {\bigwedge}_\cO^{r+e}H^1_\Sigma(\cO_{K,S},T)\otimes_\cO {\bigwedge}_\cO^e H^2_\Sigma(\cO_{K,S},T)^\ast \right) \ar[ru]_-{R^{(k_0)}}.&
}
$$
If we assume the Iwasawa main conjecture, then there exists a $\Lambda$-basis
$$\fz \in {\det}_\Lambda^{-1}(\rgamma_\Sigma(\cO_{K,S},\TT))$$
such that $\pi(\fz) = c$. Let $\fz_0 \in {\det}_\cO^{-1}(\rgamma_\Sigma(\cO_{K,S},T))$ be the image of $\fz$. By the commutative diagram above, we have
$$\cD(c) = R^{(k_0)}(\pi_0(\fz_0)).$$
If we further assume Conjecture \ref{leading}, then we have
$$\cD(c)=R^{(k_0)}(\eta)$$
and the injectivity of $R^{(k_0)}$ implies
$$\pi_0(\fz_0) =  \eta.$$
Since $\fz_0$ is an $\cO$-basis of ${\det}_\cO^{-1}(\rgamma_\Sigma(\cO_{K,S},T))$, we see by \cite[Proposition 2.17]{ks} that the Tamagawa number conjecture for $(M^\ast(1),\cO)$ is true. 
\end{proof}

\subsection{The anticyclotomic case}\label{anti case}

In this subsection, we give an explicit interpretation of Conjecture \ref{leading} in the setting of Example \ref{ex}(iii). 

Let $p, K, K_\infty, E, T, V,A$ be as in \S \ref{sec bd} and {\it assume Hypothesis \ref{heegner hyp}}. Let $S$ be the set of places of $K$ which divide $pN\infty$. Note that Hypothesis \ref{hyp} is satisfied (with $\cO=\ZZ_p$ and $\Sigma=\emptyset$). {\it We continue assuming Hypothesis \ref{wl}, i.e., $H^2(\cO_{K,S},\TT)$ is $\Lambda$-torsion}. 

Let $E^K$ be the quadratic twist of $E$ by $K$.
Let $\omega$ and $\omega^K$ be the N\'eron differentials of $E/\QQ$ and $E^K/\QQ$ respectively, which form a $\QQ$-basis of $\Gamma(E,\Omega_{E/K}^1)$. The formal logarithm map induces an isomorphism
$$\QQ_p \otimes_{\ZZ_p}\bigoplus_{v\mid p}E(K_v)^{\wedge}\xrightarrow{\sim} \QQ_p\otimes_\QQ \Gamma(E,\Omega_{E/K}^1)^\ast  ,$$
where $(-)^\wedge$ denotes the $p$-completion. The basis of $\QQ_p \otimes_{\ZZ_p}\bigoplus_{v\mid p}E(K_v)^{\wedge}$ corresponding to $\{\omega,\omega^{K}\}$ gives an identification 
\begin{equation}\label{identify}
\QQ_p \otimes_{\ZZ_p}\bigoplus_{v\mid p}E(K_v)^{\wedge} =\QQ_p^2. 
\end{equation}

We set
$$e:= \dim_{\QQ_p}(H^2(\cO_{K,S},V)) \text{ and }s:={\rm rk}_\ZZ (E(K)).$$
Note that the proof of Lemma \ref{AC lemma} shows that
$$e= \begin{cases}
s-2 &\text{if $r^+,r^- >0$,}\\
s-1 &\text{if either $r^+ =0 $ or $r^-=0$},
\end{cases}$$
where $r^+:={\rm rk}_\ZZ(E(\QQ))$ and $r^-:={\rm rk}_\ZZ(E^K(\QQ))$. 

Recall from \cite[(5.1.1)]{ks} that there is a long exact sequence
\begin{multline}\label{ks long}
0\to \QQ_p\otimes_\ZZ E(K) \to H^1(\cO_{K,S},V)\to \QQ_p\otimes_{\ZZ_p} \bigoplus_{v\mid p} E(K_v)^{\wedge,\ast}\\
 \to \QQ_p\otimes_\ZZ E(K)^\ast \to H^2(\cO_{K,S},V)\to 0.
\end{multline}
(Later we need an integral version of this sequence. See Lemma \ref{poitoutate} below.) From this and the identification (\ref{identify}), we obtain an isomorphism
\begin{multline}\label{important isom}
\QQ_p\otimes_\ZZ \left( {\bigwedge}_\ZZ^s E(K)\otimes_\ZZ {\bigwedge}_\ZZ^s E(K) \right) = \QQ_p \otimes_\ZZ \left( {\det}_\ZZ (E(K)) \otimes_\ZZ {\det}_\ZZ^{-1} (E(K)^\ast)\right) \\
\simeq {\det}_{\QQ_p}(H^1(\cO_{K,S},V)) \otimes_{\QQ_p} {\det}_{\QQ_p}^{-1}(H^2(\cO_{K,S},V)) = {\bigwedge}_{\QQ_p}^{e+2} H^1(\cO_{K,S},V) \otimes_{\QQ_p} {\bigwedge}_{\QQ_p}^e H^2(\cO_{K,S},V)^\ast.
\end{multline}

Let
$$\beta^{(1)}=\beta^{(1)}(\rgamma(\cO_{K,S}, \TT)): H^1(\cO_{K,S},T) \to H^2(\cO_{K,S},T)\otimes_{\ZZ_p} Q^1,$$
$$\beta^{(k)}=\beta^{(k)}(\rgamma(\cO_{K,S},\TT)): \ker \beta^{(k-1)} \to \coker \beta^{(k-1)}\otimes_{\ZZ_p} Q^1$$
be the derived Bockstein maps for $\rgamma(\cO_{K,S},\TT)$ (see Definition \ref{main def}). 
Let $k_0$ be as in Definition \ref{def k0}. We set
$$\varrho:=e+\sum_{k=1}^{k_0-1}{\rm rk}_{\ZZ_p}(\coker \beta^{(k)}).$$
The $k_0$-th derived Bockstein regulator map in (\ref{def der map}) becomes a map
$$R^{(k_0)}:  {\bigwedge}_{\QQ_p}^{e+2} H^1(\cO_{K,S},V)\otimes_{\QQ_p}{\bigwedge}_{\QQ_p}^e H^2(\cO_{K,S},V)^\ast \to  {\bigwedge}_{\QQ_p}^2 H^1(\cO_{K,S},V) \otimes_{\ZZ_p}  Q^{\varrho}.$$

\begin{definition}\label{def anti boc}
Consider the composition map
\begin{eqnarray*}
\QQ_p\otimes_{\ZZ}\left({\bigwedge}_\ZZ^{s} E(K) \otimes_\ZZ {\bigwedge}_\ZZ^{s}E(K) \right) &\stackrel{(\ref{important isom})}{\simeq}&   {\bigwedge}_{\QQ_p}^{e+2} H^1(\cO_{K,S},V)\otimes_{\QQ_p}{\bigwedge}_{\QQ_p}^e H^2(\cO_{K,S},V)^\ast \\
& \xrightarrow{R^{(k_0)}}&   {\bigwedge}_{\QQ_p}^2 H^1(\cO_{K,S},V) \otimes_{\ZZ_p}  Q^{\varrho}.
\end{eqnarray*}
We define the {\it anticyclotomic derived Bockstein regulator}
$${\rm Reg}^{\rm Boc} \in  {\bigwedge}_{\QQ_p}^2 H^1(\cO_{K,S},V) \otimes_{\ZZ_p}  Q^{\varrho}$$
to be the image of 
$$(x_1\wedge\cdots \wedge x_{s})\otimes (x_1 \wedge\cdots \wedge x_{s}) \in \QQ_p\otimes_{\ZZ}\left({\bigwedge}_\ZZ^{s} E(K) \otimes_\ZZ {\bigwedge}_\ZZ^{s}E(K) \right) $$
under the above map, where $\{x_1,\ldots,x_{s}\}$ is a basis of $E(K)_{\rm tf}$. 
\end{definition}

\begin{remark}\label{rho remark}
Let
$$R_{K_\infty}^{\rm Boc} \in {\bigwedge}_{\QQ_p}^2 H^1(\cO_{K,S},V) \otimes_{\ZZ_p} Q^e$$
be the anticyclotomic Bockstein regulator defined in \cite[Def. 5.26]{ks}. If $k_0=1$, then one sees by definition that $\varrho =e$ and ${\rm Reg}^{\rm Boc} = R_{K_\infty}^{\rm Boc}$. We remark that $R_{K_\infty}^{\rm Boc}$ can vanish, whereas ${\rm Reg}^{\rm Boc}$ is always non-zero. In fact, we have the implication
$$ |r^+- r^-| >1 \Rightarrow R_{K_\infty}^{\rm Boc} =0.$$
This is verified as follows. If $R_{K_\infty}^{\rm Boc} \neq 0$, then by definition we have ${\rm rk}_{\ZZ_p}(\coker \beta^{(1)})=0$. This implies that ${\rm rk}_{\ZZ_p}(\ker \beta^{(1)}) =2$. Note that we have a commutative diagram
$$\xymatrix{
\widetilde H^1_f(K,T) \ar[r]^-{\widetilde \beta^{(1)}} \ar@{^{(}->}[d] & \widetilde H^2_f(K,T)\otimes_{\ZZ_p}Q^1 \ar[d] \\
H^1(\cO_{K,S},T) \ar[r]_-{\beta^{(1)}} & H^2(\cO_{K,S},T)\otimes_{\ZZ_p}Q^1,
}$$
where $\widetilde \beta^{(1)}$ is the Bockstein map for $\widetilde \rgamma_f(K,\TT)$ (see \S \ref{review height}) and the vertical arrows are induced by the triangle (\ref{triangle}). From this, we have
$\ker \widetilde \beta^{(1)} \subset \ker \beta^{(1)}$ and so 
$${\rm rk}_{\ZZ_p}(\ker \widetilde \beta^{(1)}) \leq 2. $$
By noting that ${\rm rk}_{\ZZ_p}(\ker \widetilde \beta^{(1)})\geq |r^+-r^-|$ (by Proposition \ref{higher height}(ii)) and that $|r^+-r^-|$ is odd (by the Heegner hypothesis and the parity conjecture), we conclude that $|r^+-r^-| =1$. 
\end{remark}

As an analogue of the conjecture of Mazur and Bertolini-Darmon (see Remark \ref{MBD}), we suggest the following conjecture. 

\begin{conjecture}\label{MBD analogue}
We have
$${\rm rk}_{\ZZ_p}(\coker \beta^{(k)}) = \begin{cases}
|r^+-r^-|-1 & \text{if $k=1$,}\\
0 &\text{if $k\geq 2$.}
\end{cases}$$
\end{conjecture}
\begin{remark}
If Conjecture \ref{MBD analogue} is true, then we have
$$k_0 = \begin{cases}
1 &\text{if $|r^+-r^-| =1$,}\\
2 &\text{if $|r^+-r^-| >1$}
\end{cases}$$
and 
$$\varrho = e+ |r^+-r^-|-1 = \begin{cases}
2(\max\{r^+,r^-\} -1) -1 &\text{if $r^+,r^->0$,}\\
2(\max\{r^+,r^-\} -1) &\text{if either $r^+=0$ or $r^-=0$.}
\end{cases}$$
\end{remark}

In this paper, we do not study Conjecture \ref{MBD analogue} in detail. All the results below are independent of its validity. 

Let $L(E/K,s)$ denote the Hasse-Weil $L$-function for $E/K$. The $S$-truncated $L$-function (i.e., the Euler factors at all $v\in S$ are removed) is denoted by $L_S(E/K,s)$. The leading term of $L_S(E/K,s)$ at $s=1$ is denoted by $L_S^\ast(E/K,1)\in \RR^\times$. We fix an embedding $\RR \hookrightarrow \CC_p$ and regard $L_S^\ast(E/K,1) \in \CC_p^\times$. Let $D_K$ be the discriminant of $K$. Let $\Omega_{E/K}$ and $R_{E/K}$ be the N\'eron period and the N\'eron-Tate regulator for $E/K$ respectively. 

Let $z^{\rm Hg} \in Q(\Lambda) \otimes_\Lambda {\bigcap}_\Lambda^2 H^1(\cO_{K,S},\TT)$ be the $\Lambda$-adic Heegner element in Example \ref{ex}(iii). Conjecture \ref{leading}(ii) is explicitly interpreted as follows. 

\begin{proposition}\label{interpretation}
Conjecture \ref{leading}(ii) for $z^{\rm Hg}$ holds if and only if we have an equality
\begin{equation}\label{derivative formula}
\cD(z^{\rm Hg})= \frac{L_S^\ast(E/K,1)\sqrt{|D_K|}}{\Omega_{E/K} \cdot R_{E/K}}\cdot {\rm Reg}^{\rm Boc}\text{ in }\CC_p\otimes_{\ZZ_p}{\bigwedge}_{\ZZ_p}^2 H^1(\cO_{K,S},T)\otimes_{\ZZ_p}Q^{\varrho}.
\end{equation}
\end{proposition}

\begin{proof}
This follows immediately from the explicit description of the extended special element $\eta$ given in the proof of \cite[Proposition 5.27]{ks}. 
\end{proof}

Let ${\rm Eul}_S$ be the product of Euler factors at finite places $v$ in $S$ (so that ${\rm Eul}_S \cdot L^\ast(E/K,1)=L_S^\ast(E/K,1)$). Since the Birch and Swinnerton-Dyer formula predicts the equality
$$\frac{L_S^\ast(E/K,1)\sqrt{|D_K|}}{\Omega_{E/K}\cdot R_{E/K}} = {\rm Eul}_S\cdot \frac{\#\sha(E/K)\cdot {\rm Tam}(E/K)}{\# E(K)_{\rm tors}^2},$$
Proposition \ref{interpretation} suggests the following conjecture. 

\begin{conjecture}\label{alg conj}
There exists $u \in \ZZ_p^\times$ such that
$$\cD(z^{\rm Hg})= u\cdot {\rm Eul}_S\cdot \#\sha(E/K)[p^\infty]\cdot {\rm Tam}(E/K)\cdot {\rm Reg}^{\rm Boc}. $$
\end{conjecture}




We show that Conjecture \ref{alg conj} is a consequence of the Heegner point main conjecture. 

\begin{theorem}\label{alg from IMC}
The Heegner point main conjecture (Conjecture \ref{IMC}) implies Conjecture \ref{alg conj}. 
\end{theorem}

\begin{remark}
Theorem \ref{alg from IMC} is a refinement of \cite[Theorem 5.28]{ks}. 
\end{remark}


We also prove the following.

\begin{theorem}\label{bsd strategy}
Assume the Heegner point main conjecture (Conjecture \ref{IMC}). If we further assume the equality (\ref{derivative formula}), then the $p$-part of the Birch and Swinnerton-Dyer formula is true, i.e., there exists $u \in \ZZ_p^\times$ such that
$$\frac{L^\ast(E/K,1)\sqrt{|D_K|}}{\Omega_{E/K}\cdot R_{E/K}} = u\cdot \#\sha(E/K)[p^\infty]\cdot {\rm Tam}(E/K).$$
\end{theorem}

\begin{remark}
Theorem \ref{bsd strategy} is a refinement of \cite[Theorem 5.30]{ks}. 
\end{remark}

The rest of this paper is devoted to the proofs of Theorems \ref{alg from IMC} and \ref{bsd strategy}.

Theorem \ref{bsd strategy} is actually a direct consequence of Theorem \ref{strategy}. In fact, we know that the Heegner point main conjecture (Conjecture \ref{IMC}) is equivalent to Conjecture \ref{ks main conj} for $z^{\rm Hg}$, i.e.,
\begin{equation}\label{classical IMC}
{\rm char}_\Lambda\left( {\bigcap}_\Lambda^2 H^1(\cO_{K,S},\TT)/\Lambda\cdot z^{\rm Hg}\right) = {\rm char}_\Lambda(H^2(\cO_{K,S},\TT)). 
\end{equation}
(See \cite[Theorem 5.18]{ks}.) 
Since the Tamagawa number conjecture in this case is equivalent to the $p$-part of the Birch and Swinnerton-Dyer formula, Theorem \ref{bsd strategy} follows from Theorem \ref{strategy}.

To prove Theorem \ref{alg from IMC}, we give some preliminaries. For $v\in S$, let $H^1_f(K_v, -) \subset H^1(K_v,-)$ denote the Bloch-Kato local condition and set $H^1_{/f}(K_v,-):=H^1(K_v,-)/H^1_f(K_v,-)$ (see \cite[\S 3]{BK} or \cite[\S 1.3]{R}). Recall that the Bloch-Kato Selmer group is defined by 
$$H^1_f(K,-):=\ker \left( H^1(\cO_{K,S},-) \to \bigoplus_{v\in S} H^1_{/f}(K_v,-)\right).$$

In the following, we set
$$E(K_p):=\bigoplus_{v\mid p} E(K_v).$$ 
For an abelian group $X$, we denote its $p$-completion by $X^\wedge:=\varprojlim_n X/p^n X$. 
Recall that we set $A:=V/T = E[p^\infty]$. 

\begin{lemma}\label{poitoutate}
There is a canonical exact sequence
\begin{multline*}
0\to \ZZ_p\otimes_\ZZ E(K) \to H^1(\cO_{K,S},T) \to E(K_p)^{\wedge,\ast} \\
\to H^1_f(K,A)^\vee \to H^2(\cO_{K,S},T) \to \bigoplus_{v\in S}E(K_v)[p^\infty]^\vee \to 0. 
\end{multline*}
\end{lemma}

\begin{proof}
By the Poitou-Tate duality (see \cite[(5.1.6)]{nekovar}), we have a long exact sequence
\begin{multline*}
H^1(\cO_{K,S},T) \to \bigoplus_{v\in S} H^1(K_v,T) \to H^1(\cO_{K,S}, T^\vee(1))^\vee \\
 \to H^2(\cO_{K,S},T)\to \bigoplus_{v\in S}H^2(K_v,T)\to H^0(K, T^\vee(1))^\vee .
\end{multline*}
From this, one can deduce the following exact sequence:
\begin{multline*}
0\to H^1_f(K,T)\to H^1(\cO_{K,S},T) \to \bigoplus_{v\in S}H^1_{/f}(K_v,T)\\
\to H^1_f(K,T^\vee(1))^\vee \to H^2(\cO_{K,S},T)\to \bigoplus_{v\in S}H^2(K_v,T)\to H^0(K,T^\vee(1))^\vee.
\end{multline*}
By the Weil pairing, we can identify $T^\vee(1)$ with $A=E[p^\infty]$. Since we assume $E(K)[p]=0$ and $\# \sha(E/K)[p^\infty]< \infty$ (see Hypothesis \ref{heegner hyp}), we have
$$H^0(K,A)=0 \text{ and }H^1_f(K,T) \simeq  \ZZ_p\otimes_\ZZ E(K).$$
Also, for $v\in S$, we have
$$H^2(K_v,T) \simeq H^0(K_v,A)^\vee = E(K_v)[p^\infty]^\vee$$
by local duality. 
The lemma follows by noting that 
$$H^1_{/f}(K_v,T) = \begin{cases}
0 &\text{ if $v\nmid p$,}\\
E(K_v)^{\wedge,\ast} &\text{ if $v\mid p$.}
\end{cases}$$
\end{proof}



Let
$$\log: \QQ_p \otimes_{\ZZ_p}E(K_p)^{\wedge}\xrightarrow{\sim} \QQ_p\otimes_\QQ \Gamma(E,\Omega_{E/K}^1)^\ast $$
be the formal logarithm map. We define the ``Euler factor at $p$" by 
$${\rm Eul}_p:= \frac{1}{p^2}\prod_{v\mid p}\# E(\FF_v),$$
where $\FF_v$ denotes the residue field of $v$. (Note that $E$ has good reduction at $p$ by Hypothesis \ref{heegner hyp}.) 

\begin{lemma}\label{euler p}
The image of the map
\begin{multline*}
{\det}_{\ZZ_p}^{-1}\left(E(K_p)^\wedge \right) \hookrightarrow \QQ_p\otimes_{\ZZ_p} {\det}_{\ZZ_p}^{-1}\left(E(K_p)^\wedge \right)\\
  \stackrel{\log}{\simeq} \QQ_p\otimes_\QQ {\det}_\QQ^{-1}(\Gamma(E,\Omega_{E/K}^1)^\ast) = \QQ_p\otimes_\QQ {\bigwedge}_\QQ^2 \Gamma(E,\Omega_{E/K}^1)
\end{multline*}
is $\ZZ_p\cdot {\rm Eul}_p \cdot \omega \wedge \omega^K$. 
\end{lemma}

\begin{proof}
For $v\mid p$, we have a canonical exact sequence
$$0\to E_1(K_v) \to E(K_v) \to E(\FF_v)\to 0.$$
From this, we obtain a canonical isomorphism
$${\det}_{\ZZ_p}^{-1}\left(  E(K_p)^\wedge \right) \simeq  \left(\prod_{v\mid p}\# E(\FF_v)\right)\cdot {\det}_{\ZZ_p}^{-1}\left( E_1(K_p)\right)\subset  \QQ_p\otimes_{\ZZ_p} {\det}_{\ZZ_p}^{-1}\left(E(K_p)^\wedge \right).$$
Since $p$ is odd, the formal logarithm associated to $\omega$ induces an isomorphism $ E_1(K_p)\xrightarrow{\sim} p\ZZ_p\otimes_\ZZ \cO_{K}$. So we obtain an isomorphism
$${\det}_{\ZZ_p}^{-1}\left(E_1(K_p)\right) \simeq \ZZ_p\cdot \frac{1}{p^2} \omega\wedge \omega^K \subset \QQ_p\otimes_\QQ {\bigwedge}_\QQ^2 \Gamma(E,\Omega_{E/K}^1).$$
The claim follows from these isomorphisms. 
\end{proof}

We now prove Theorem \ref{alg from IMC}.

\begin{proof}[Proof of Theorem \ref{alg from IMC}]
We use the commutative diagram in Theorem \ref{thm descent}:
$$
\scriptsize
\xymatrix@C=12pt@R=30pt{
\displaystyle {\det}_\Lambda^{-1}( \rgamma(\cO_{K,S},\TT)) \ar[r]^-{\pi } \ar@{->>}[dd] &\displaystyle I^{\varrho}\cdot {\bigwedge}_{\Lambda_I}^2 H^1(\cO_{K,S},\TT)_I \ar[rd]^-{\cD} & \\
 & & \displaystyle  \QQ_p\otimes_{\ZZ_p}\left( {\bigwedge}_{\ZZ_p}^2H^1(\cO_{K,S},T)\otimes_{\ZZ_p} Q^\varrho \right) \\
\displaystyle {\det}_{\ZZ_p}^{-1}( \rgamma(\cO_{K,S},T)) \ar[r]_-{\pi_0    } &\displaystyle \QQ_p\otimes_{\ZZ_p}\left( {\bigwedge}_{\ZZ_p}^{e+2}H^1(\cO_{K,S},T)\otimes_{\ZZ_p} {\bigwedge}_{\ZZ_p}^e H^2(\cO_{K,S},T)^\ast \right) \ar[ru]_-{ R^{(k_0)}}.&
}
$$
We note that the Heegner point main conjecture formulated as (\ref{classical IMC}) is equivalent to the following: there is a $\Lambda$-basis
$$\fz \in {\det}_\Lambda^{-1}(\rgamma(\cO_{K,S},\TT))$$
such that $\pi(\fz) = z_\infty^{\rm Hg}$ (see \cite[Proposition 3.10]{ks}). Note that the image of $\fz$ in ${\det}_{\ZZ_p}^{-1}(\rgamma(\cO_{K,S},T))$ is a $\ZZ_p$-basis. 
Let
$$\delta: \QQ_p\otimes_{\ZZ_p}\left( {\bigwedge}_{\ZZ_p}^{e+2}H^1(\cO_{K,S},T)\otimes_{\ZZ_p} {\bigwedge}_{\ZZ_p}^e H^2(\cO_{K,S},T)^\ast \right)\xrightarrow{\sim} \QQ_p\otimes_\ZZ \left( {\bigwedge}_\ZZ^{s} E(K) \otimes_\ZZ {\bigwedge}_\ZZ^{s} E(K)\right)$$
be the inverse of (\ref{important isom}). By the commutative diagram above and the definition of ${\rm Reg}^{\rm BD}$ (see Definition \ref{def anti boc}), it is sufficient to prove
\begin{equation}\label{compute h2}
\im (\delta \circ \pi_0) = {\rm Eul}_S\cdot \# \sha(E/K)[p^\infty]\cdot {\rm Tam}(E/K)\cdot \ZZ_p \otimes_\ZZ \left( {\bigwedge}_\ZZ^{s} E(K) \otimes_\ZZ {\bigwedge}_\ZZ^{s} E(K)\right).
\end{equation}

By Lemma \ref{poitoutate}, we have a canonical isomorphism
\begin{multline}\label{compute 0}
{\det}_{\ZZ_p}^{-1}(\rgamma(\cO_{K,S},T)) \\
\simeq {\det}_{\ZZ_p}(\ZZ_p\otimes_\ZZ E(K)) \otimes_{\ZZ_p} {\det}_{\ZZ_p}\left(  E(K_p)^{\wedge,\ast}\right) \otimes_{\ZZ_p} {\det}_{\ZZ_p}^{-1}(H^1_f(K,A)^\vee) \otimes_{\ZZ_p} {\det}_{\ZZ_p}^{-1}\left( \bigoplus_{v\in S} E(K_v)[p^\infty]^\vee\right).
\end{multline}
We shall calculate each term on the right hand side. 

First, since $E(K)[p]=0$, we have
\begin{equation}\label{obvious isom}
{\det}_{\ZZ_p}(\ZZ_p\otimes_\ZZ E(K)) = \ZZ_p\otimes_\ZZ {\bigwedge}_\ZZ^{s} E(K). 
\end{equation}

Next, by Lemma \ref{euler p} and the identification (\ref{identify}), we have an isomorphism
\begin{eqnarray}\label{compute 1}
{\det}_{\ZZ_p}\left(E(K_p)^{\wedge,\ast}\right) \otimes_{\ZZ_p} {\det}_{\ZZ_p}^{-1}\left( E(K_p)[p^\infty]^\vee\right) &=& {\det}_{\ZZ_p}(\rhom_{\ZZ_p}(E(K_p)^\wedge , \ZZ_p)) \\ 
&= & {\det}_{\ZZ_p}^{-1}(E(K_p)^\wedge)  \nonumber \\ 
&\simeq& {\rm Eul}_p\cdot \ZZ_p.\nonumber
\end{eqnarray}

Then, we consider the dual Selmer group $H^1_f(K,A)^\vee$. Since we have a canonical exact sequence
$$0\to \sha(E/K)[p^\infty]^\vee \to H^1_f(K,A)^\vee \to \ZZ_p\otimes_\ZZ E(K)^\ast \to 0,$$
we obtain an isomorphism
\begin{equation}\label{compute 2}
{\det}_{\ZZ_p}^{-1}(H^1_f(K,A)^\vee) \simeq \# \sha(E/K)[p^\infty]\cdot \ZZ_p \otimes_\ZZ {\bigwedge}_\ZZ^{s} E(K) \subset \QQ_p\otimes_\ZZ {\bigwedge}_\ZZ^{s}E(K). 
\end{equation}

Finally, we calculate ${\det}_{\ZZ_p}^{-1}(E(K_v)[p^\infty]^\vee )$ for $v \nmid p$. Let ${\rm Eul}_N$ be the product of Euler factors at $v\mid N$, which satisfies ${\rm Eul}_S = {\rm Eul}_p \cdot {\rm Eul}_N$. We have $E_0(K_v)^\wedge = E^{\rm ns}(\FF_v)\otimes_\ZZ \ZZ_p$ for $v\mid N$ and $\left(\prod_{v\mid N}\# E^{\rm ns}(\FF_v)\right)\cdot  \ZZ_p = {\rm Eul}_N\cdot \ZZ_p$. Thus we have
\begin{equation}\label{compute 3}
{\det}_{\ZZ_p}^{-1}\left( \bigoplus_{v\mid N} E(K_v)[p^\infty]^\vee\right) \simeq \left( \prod_{v\mid N} \# E(K_v)[p^\infty] \right) \ZZ_p= {\rm Eul}_N \cdot {\rm Tam}(E/K)\cdot \ZZ_p \subset \QQ_p.
\end{equation}

Combining (\ref{compute 0}) with (\ref{obvious isom}), (\ref{compute 1}), (\ref{compute 2}) and (\ref{compute 3}), we obtain an isomorphism
$${\det}_{\ZZ_p}^{-1}(\rgamma(\cO_{K,S},T)) \simeq {\rm Eul}_S\cdot \# \sha(E/K)[p^\infty]\cdot {\rm Tam}(E/K)\cdot \ZZ_p \otimes_\ZZ \left( {\bigwedge}_\ZZ^{s} E(K) \otimes_\ZZ {\bigwedge}_\ZZ^{s} E(K)\right).$$
By definition, this isomorphism coincides with the map induced by $\delta \circ \pi_0$. This shows the desired equality (\ref{compute h2}) and we have completed the proof of Theorem \ref{alg from IMC}. 
\end{proof}


\begin{acknowledgments}
The author would like to thank Francesc Castella for valuable comments on an earlier version of this paper and for sending him the preprint \cite{CHKLL}. He would also like to thank David Burns and Masato Kurihara for their constant encouragement and many helpful discussions on related topics. 
\end{acknowledgments}

\end{document}